\documentclass[10pt,a4paper]{amsart}
\usepackage{amsfonts}
\usepackage{amsthm}
\usepackage{amsmath}
\usepackage{amscd}
\usepackage[latin2]{inputenc}
\usepackage{t1enc}
\usepackage[mathscr]{eucal}
\usepackage{indentfirst}
\usepackage{graphicx}
\usepackage{graphics}
\usepackage{pict2e}
\usepackage{epic}
\usepackage[normalem]{ulem}
\numberwithin{equation}{section}
\usepackage{epstopdf}
\usepackage{verbatim}
\usepackage{amssymb}
\usepackage{tikz-cd}
\usepackage{stmaryrd}
\usepackage{accents}

\usepackage{hyperref}
\AtBeginDocument{}

\theoremstyle{plain}
\newtheorem{Th}{Theorem}[section]

\newtheorem{Prop}[Th]{Proposition}

\theoremstyle{definition}

\newtheorem{Rem}[Th]{Remark}
\newtheorem{?}[Th]{Problem}

\def\al{\alpha}

\def\w{\wedge}
\def\R{\mathbb{R}}
\def\C{\mathbb{C}}

\def\Lm{\Lambda}

\def\om{\omega}
\def\Om{\Omega}

\def\ip{\raise1pt\hbox{\large$\lrcorner$}\>}

\DeclareMathOperator{\vol}{vol}

\begin{document}
	
\title{Einstein metrics on bundles over hyperK\"ahler manifolds}

\author[U. Fowdar]{Udhav Fowdar}
	
\address{IMECC - Unicamp \\
	Rua S{\'e}rgio Buarque de Holanda, 651 - Campinas \\
	SP, 13083-859\\
	Brazil}
	
\email{udhav@unicamp.com}
		
\keywords{Hermitian geometry, Quaternion-K\"ahler geometry,  Torus action} 
\subjclass[2010]{53C10, 53C26, 53C55}	
	
\begin{abstract} 
We construct explicit examples of quaternion-K\"ahler and hypercomplex structures on bundles over hyperK\"ahler manifolds. We study the infinitesimal symmetries of these examples and the associated Galicki-Lawson quaternion-K\"ahler moment map. By performing the QK reduction we produce several explicit QK metrics. Moreover we are led to a new proof of a  hyperK\"ahler/quaternion-K\"ahler type correspondence. We also give examples of other Einstein metrics and balanced Hermitian structures on these bundles.
\end{abstract}

\maketitle
\tableofcontents

\section{Introduction}

When searching for examples of Einstein metrics, a natural starting point is to look for examples on bundles over other Einstein manifolds.
For instance, this strategy was used by Calabi to construct the first examples of Calabi-Yau and hyperK\"ahler metrics on bundles over K\"ahler-Einstein manifolds \cite{Calabi1957} and by Bryant-Salamon to construct the first examples of $G_2$ and $Spin(7)$ metrics on bundles over $S^3$, $S^4$ and $\C \mathbb{P}^2$ \cite{Bryant1989}. In each of these cases the total spaces were endowed with a natural family of $G$-structures (where $G=SU(n),Sp(n),G_2$ or $Spin(7)$) and the desired Einstein metrics were found by solving suitable systems of ordinary differential equations. 
In this paper, motivated by examples found by Gibbons-L\"u-Pope-Stelle in \cite{Gibbons01}, we consider the Einstein equations on certain natural bundles over hyperK\"ahler manifolds determined by the hyperK\"ahler $2$-forms. The rather simple nature of the structure equations also allows us to find other geometrically interesting structures namely hypercomplex and balanced Hermitian manifolds. We shall now elaborate on the results in this paper.\\

\noindent\textbf{Outline of paper.} 
Throughout this article, we denote by $(M^{4n},\sigma_1,\sigma_2,\sigma_3,g_M)$ a hyperK\"ahler $4n$-manifold with $\sigma_i$ denoting the triple of K\"ahler forms and $g_M$ the hyperK\"ahler metric.
If $[\sigma_1] \in H^2(M,\R)$ defines an integral cohomology class i.e. $[\sigma_1] \in H^2(M,2\pi\mathbb{Z})$, we denote by $M^{4n+1}_{\al}$ the total space of the $S^1$ bundle determined by $[\sigma_1]$ and let $\al$ be a connection $1$-form such that 
\begin{equation}\label{connection1}
	d\al=\sigma_1.
\end{equation}	
The complex line bundle associated to this $U(1)$ bundle is the so-called pre-quantum line bundle.
If additionally, $[-\sigma_2] \in H^2(M,2\pi\mathbb{Z})$, we denote by $M^{4n+2}_{\al, \xi}$ the total space of the $\mathbb{T}^2$ bundle determined by $[\sigma_1], [-\sigma_2]$ and let $\xi$ be a connection $1$-form such that 
\begin{equation}\label{connection2}
d\xi=-\sigma_2. 
\end{equation}
Likewise, if $[-\sigma_3] \in H^2(M,2\pi\mathbb{Z})$, we denote by $M^{4n+3}_{\al,\xi,\eta}$ the $\mathbb{T}^3$ bundle determined by $[\sigma_1], [-\sigma_2], [-\sigma_3] $ and endowed with a connection $1$-form $\eta$ such that 
\begin{equation}\label{connection3}
d\eta=-\sigma_3.
\end{equation}
Before stating our first main result we first recall the following well-known warped product Einstein metric
\begin{equation}
	g_Q:=dt^2+e^{2t}g_{M}
\end{equation}
defined on $Q^{4n+1}=\R_t \times M^{4n}$. More generally, one can define an Einstein metric of this form by replacing $g_M$ by any Ricci-flat metric cf. \cite[p. 267]{Besse2008}. When $M$ is a hyperK\"ahler manifold however then we show that this Einstein metric can in fact be extended as follows: 
\begin{Th}\label{maintheorem} 
The manifold $P^{4n+2}=\R_t \times M^{4n+1}_\al$ admits a K\"ahler-Einstein metric 
\begin{equation}g_P:=dt^2+4e^{4t}(\al^2)+e^{2t}g_{M}\label{kahlereinstein}\end{equation}
with associated K\"ahler $2$-form given by
\begin{equation}\om_P:=2e^{2t}dt \w \al + e^{2t}\sigma_1.\end{equation}
The manifold  $L^{4n+3}=\R_t \times M^{4n+2}_{\al, \xi}$ admits an Einstein metric 
\begin{equation}g_L:=dt^2+4e^{4t}(\al^2+\xi^2)+e^{2t}g_{M}.\end{equation}
The manifold $N^{4n+4}=\R_t \times M^{4n+3}_{\al,\xi,\eta}$ admits a quaternion-K\"ahler metric 
\begin{equation}g_N:=dt^2+4e^{4t}(\al^2+\xi^2+\eta^2)+e^{2t}g_{M}\label{qkmetric}\end{equation}
with the associated quaternion-K\"ahler $4$-form given by
\begin{equation}
\Om_{N}:=\frac{1}{2}(\om_1 \w \om_1+\om_2 \w \om_2+\om_3 \w \om_3),\label{QK4form}
\end{equation}
where
\begin{gather}
	\begin{cases}
\om_1 := e^{2t}\sigma_1+(2 e^{2t}\xi) \wedge (2 e^{2t}\eta) + dt \wedge  (2e^{2t}\al), \\
\om_2	:= e^{2t}\sigma_2+(2 e^{2t}\xi) \wedge (dt) + (2e^{2t}\al) \wedge  (2 e^{2t}\eta), \\
\om_3 := e^{2t}\sigma_3+(2 e^{2t}\xi) \wedge (2 e^{2t}\al) + (2 e^{2t}\eta) \wedge dt.
\end{cases}
\end{gather}
Furthermore, the almost complex structures $I_i$ defined by $g_{N}(I_i(\cdot),\cdot)=\om_i(\cdot,\cdot)$ are all integrable i.e. $N^{4n+4}$ is a hypercomplex manifold.  In all the cases, the Einstein metrics have negative scalar curvature and are complete if and only if $g_M$ is.
\end{Th}
It is worth emphasising that $N^{4n+4}$ is not a hyperK\"ahler manifold since $\om_i$ are not closed.
When $M=\mathbb{T}^4$, the total spaces $P^6$, $L^7$ and $N^8$ in Theorem \ref{maintheorem} with their Einstein metrics have been identified as certain non-unimodular solvable Lie groups in \cite{Gibbons01}. 
The latter QK metric on $N^8$ was rediscovered in \cite[(4.24)]{IvanovQKContact2010}, more generally when $M=\mathbb{T}^{4n}$, by studying quaternionic contact structures introduced by Biquard \cite{Biquard2000}. In general, however, the Einstein metrics arising from Theorem \ref{maintheorem} are not homogeneous unless $M^{4n}$ is flat.

The fact that the underlying QK structure is hypercomplex is not too surprising since it is well-known that infinitesimal symmetries of QK manifolds correspond to integrable almost complex structures cf. \cite[Prop. 3.1]{GAMBIOLI2015146} and here we have that $\Om_{N}$ is $\mathbb{T}^3$-invariant.
In \cite{Goldstein2004} Goldstein and Prokushkin construct complex manifolds which are $\mathbb{T}^2$ bundles over suitable Hermitian manifolds.
The hypercomplex structure on $N^{4n+4}$ can be viewed as a quaternionic extension of their result as follows. By definition, hyperK\"ahler manifolds have holonomy group contained in $Sp(n)$. Thus, the space of $2$-forms on $M^{4n}$ splits into irreducible $Sp(n)$-modules as
\begin{equation}
\Lm^2(M) \cong \mathfrak{sp}(1) \oplus \mathfrak{sp}(n) \oplus (\mathfrak{sp}(1) \oplus \mathfrak{sp}(n))^\perp,\label{spnsp1splitting}
\end{equation}
where $\mathfrak{sp}(1)\cong \langle \sigma_1, \sigma_2, \sigma_3\rangle$ is a trivial vector bundle (in fact, flat with respect to the Levi-Civita connection) and $\mathfrak{sp}(n)$ is the holonomy algebra of $M$. If $\gamma \in \mathfrak{sp}(n) $ is a closed $2$-form such that $[\gamma]\in H^2(M,2\pi \mathbb{Z})$, let $\nu$ be the connection $1$-form such that 
\[d\nu=\gamma\]
on the associated $\mathbb{T}^4$ bundle $M^{4n+4}_{\al,\xi,\eta,\nu} \to M^{4n}$. Then we have
\begin{Th}\label{hypercomplextheorem}
On $M^{4n+4}_{\al,\xi,\eta,\nu}$ the data
\begin{gather}
	\check{g}:=g_M + \al^2 + \xi^2 + \eta^2 + \nu^2,\\
	\begin{cases}
	\check{\om}_1 :=\sigma_1+\xi \wedge \eta + \nu \wedge  \al, \\
	\check{\om}_2	:=\sigma_2+\xi \wedge \nu + \al \wedge  \eta, \\
	\check{\om}_3 :=\sigma_3+\xi \wedge \al + \eta \wedge \nu ,
	\end{cases}
\end{gather}
defines a hyper-Hermitian $Sp(n+1)$-structure, where we recall that the curvature forms are given by $d\al=\sigma_1$, $d\xi=-\sigma_2$, $d\eta=-\sigma_3$ and $d\nu=\gamma$.
\end{Th}
The hypercomplex structure on $N^{4n+4}$ is  then (locally) just a special case when $\gamma=0$ with the trivial $S^1$ bundle determined by $[\gamma]$ lifted to $\R_t$. 
We also show that there is a different hypercomplex extension of the result of Goldstein and Prokushkin in \cite{Goldstein2004}, see Theorem \ref{extendedhypercomplextheorem} below.

Theorems \ref{maintheorem}, \ref{hypercomplextheorem}, \ref{extendedhypercomplextheorem} are all proved in Section \ref{proofs}. In Sections \ref{examples}, \ref{symmetriessection}, \ref{HKQKsection} and \ref{twistorsection} we then study the geometry of these examples in detail.

In Section \ref{examples} we construct a large class of examples that arise from Theorems \ref{maintheorem}, \ref{hypercomplextheorem} and \ref{extendedhypercomplextheorem} by choosing suitable HK manifold $M^{4n}$. As already mentioned above when $M=\mathbb{T}^{4}$, the total spaces $P^6$, $L^7$ and $N^{8}$ correspond to certain solvable Lie groups and as such these examples admit lots of symmetries. In general, however, the only Killing vector fields present in the examples of Theorem \ref{maintheorem} correspond to the vertical vector fields of the $\mathbb{T}^3$ torus fibres.

In Section \ref{symmetriessection} we show that under suitable hypothesis one can lift Killing vector fields on $M^{4n}$ to Killing vector fields on the total spaces $Q^{4n+1}$, $P^{4n+2}$, $L^{4n+3}$ and $N^{4n+4}$. Furthermore, homothetic vector fields as well as permuting Killing vector fields can be lifted to Killing vector fields (see Section \ref{symmetriessection} for definitions). When lifted to $N^{4n+4}$, we show that these vector fields in fact preserve the QK $4$-form $\Om_{N}$ and hence one can consider the QK reduction cf. \cite{Galicki1988}. In particular, when $M^{4n}$ is a conical HK manifold i.e. the link is a $3$-Sasakian manifold, we consider the radial homothetic Killing vector field. By performing the QK reduction using the lifted action we obtain an explicit expression for the reduced QK metric. In the special case when $M=\R^{4n}$ we recover the QK symmetric space $\frac{Sp(n,1)}{Sp(n)Sp(1)}$:
\begin{center}
	\begin{tikzcd}[column sep=tiny]
		& (N^{4n+4},\Om_{N}) \ar[dr,"/\!\!/\!\!/\!\!/\R_{QK}"] \ar[dl,"\mathbb{R}^3_{\al,\xi,\eta}\times \R_t",swap]
		&
		&[1.5em] \\
		(\R^{4n},\om_1,\om_2,\om_3) 
		&
		& \frac{Sp(n,1)}{Sp(n)Sp(1)} 
	\end{tikzcd}
\end{center}
Here $\mathbb{R}^3_{\al,\xi,\eta}$ denotes the $\R^3$ fibres obtained by lifting the $\mathbb{T}^3$ fibres of $M^{4n+3}_{\al,\xi,\eta}$ to the universal cover so that topologically $N^{4n+4}$ is just $\R^{4n+4}.$ We also show that the HK quotient of $M^{4n}$ by tri-holomorphic Killing vector fields and the QK quotient of $N^{4n+4}$ by the lifted action commutes with the construction of Theorem \ref{maintheorem}.
Strikingly similar features were also shown to hold for the Swann bundle $\mathcal{U}(N_+^{4n})$ of a positive scalar curvature QK manifold $N^{4n}_+$ cf. \cite{Swann1991}. 

In Section \ref{HKQKsection} we show that given the HK manifold $M^{4n}$ with a permuting $S^1$ action we can lift it to a Killing action on $N^{4n+4}$. The QK quotient of $N^{4n+4}$ then yields a QK manifold $N^{4n}_{-}$ of negative scalar curvature. In the same spirit as Haydys' HK/QK correspondence via $\mathcal{U}(N_+^{4n})$ in \cite{Haydys2008}, we show that one can then reconstruct $M^{4n}$ from $N^{4n}_-$ (see Theorem \ref{HKQKcorrespondenceTheorem} for a precise statement). Thus, $N^{4n+4}$ can be thought of as an analogue of $\mathcal{U}(N_+^{4n})$. As a simple application we show that one can construct explicitly both the complex hyperbolic space and the quaternion hyperbolic space starting from $\R^4$ with suitable permuting vector fields. The correspondence we establish here was also demonstrated in \cite{Alekseevsky2013} but the key difference lies in the fact that while the latter correspondence goes via a pseudo-HK manifold with an $SU(2)$ action, our approach instead goes via the QK manifold $N^{4n+4}$ with a $\mathbb{T}^3$ action.

In Section \ref{twistorsection} we construct explicitly the twistor spaces $Z(N)$ and $Z(M_{\al,\xi,\eta,\nu})$ of the QK and hypercomplex manifolds $N^{4n+4}$ and $M_{\al,\xi,\eta,\nu}$, and we describe their relation with the twistor space $Z(M)$ of the HK manifold $M^{4n}$: 
\begin{center}
	\begin{tikzcd}[column sep=tiny]
		Z(N) \ar[d,"S^2"]\ar[rr,"T^3 \times \R"]
		& [1.5em]
		&Z(M) \ar[d,"S^2"]\\
		N \ar[rr,"T^3 \times \R"]
		&[1.5em]
		& M
	\end{tikzcd}\ \ \ \  \ \ \ \ \
	\begin{tikzcd}[column sep=tiny]
	Z(M_{\al,\xi,\eta,\nu}) \ar[d,"S^2"]\ar[rr,"T^4"]
	& [1.5em]
	&Z(M) \ar[d,"S^2"]\\
	M_{\al,\xi,\eta,\nu} \ar[rr,"T^4"]
	&[1.5em]
	& M
\end{tikzcd}
\end{center}
Twistor theory then allows us to reinterpret geometric data on the $N$ and $M_{\al,\xi,\eta,\nu}$ in terms of the complex geometry of $Z(N)$ and $Z(M_{\al,\xi,\eta,\nu})$.

In Section \ref{furtherexamples} we recall some known examples of Ricci-flat metrics on the bundles of Theorem \ref{maintheorem} with holonomy $SU(n), G_2$ and $Spin(7)$. We also show that the manifolds
$M^{4n+2}_{\xi,\eta}$, $N^{4n+4}$ and $M^{4n+4}_{\nu_1,\nu_2,\nu_3,\nu_4}$ 
admit balanced Hermitian metrics. 
In particular, we show that the metric $\check{g}$ of Theorem \ref{extendedhypercomplextheorem} is balanced for any of the hypercomplex structure.  
Balanced structures were introduced by Michelsohn in \cite{Michelsohn1982} and they also appear in the Strominger system: a system that generalises the complex Monge-Amp\`ere equations and Hermitian-Yang-Mills equations cf. \cite{Fu-Li-Yau12Balanced,MarioLectureonStromingerSystem,Li-Yau2005,Strominger1986}. 
All of our examples are very explicit and as such we expect will be useful in various applications.\\

\noindent\textbf{Acknowledgements.}
The author would like to thank Simon Salamon for introducing to him the beautiful geometry of quaternion-K\"ahler manifolds and also for many useful comments 
that helped shape this article. We would also like to thank Vicente Cort\'es for interesting remarks that led to the results in section \ref{HKQKsection}.
This work was partly supported by the S\~ao Paulo Research Foundation (FAPESP) [2021/07249-0].

\section{Preliminaries}\label{preliminaries}
The aim of this section to fix some notations and gather some standard facts about Hermitian and quaternion-K\"ahler geometry that will relevant for us. 
We refer the reader to \cite{Besse2008, Salamon1989} for proofs.
\subsection{Hermitian structures}

Let $(M^{2n},g,J,\om)$ be an almost Hermitian manifold.
The complexified space of $1$-forms splits as a $+i$ and $-i$ eigenspace under $J$ which we denote by $\Lm^{1,0}$ and $\Lm^{0,1}$ respectively. We denote by $\Lm^{p,q}$ the space of complex $(p+q)$-forms obtained by wedging $p$ elements of $\Lm^{1,0}$ with $q $ elements of $\Lm^{0,1}$.

$(M^{2n},J)$ is a complex manifold if and only if $d(\Gamma(\Lm^{1,0}))\subset \Gamma(\Lm^{1,1}\oplus \Lm^{2,0})$, or equivalently $d(\Gamma(\Lm^{n,0}))\subset \Gamma(\Lm^{n,1})$.
 $(M,g,J,\om)$ is then called a Hermitian manifold. 
If additionally, $\om$ is closed then $M$ is K\"ahler. The Hermitian manifold is said to be balanced if instead the weaker condition $d(\om^{n-1})=0$ is satisfied i.e. $\om$ is coclosed.

If there exists a global section $\tilde{\Upsilon}$ of $\Lm^{n,0}$ then the first Chern class of $(M^{2n},J)$ vanishes. If $(M,g,J,\om,\tilde{\Upsilon})$ is K\"ahler and $\tilde{\Upsilon}$ is both closed and has constant norm then $M$ is a Calabi-Yau manifold. This is equivalent to saying that the holonomy  group of $g$ is contained in $SU(n)$. In particular, $g$ is then Ricci-flat. 

We should point out that the definition of a Calabi-Yau manifold is not standard. Some people also call a complex manifold $(M,J)$ Calabi-Yau if it admits a holomorphic $(n,0)$-form cf. \cite{Fu-Li-Yau12Balanced, MarioLectureonStromingerSystem}. Henceforth, we shall always assume that $M$ is a $4n$-dimensional manifold. 

\subsection{Quaternionic structures} $(M^{4n},I_1,I_2,I_3)$ is called an almost hypercomplex manifold if the almost complex structures $I_1,I_2,I_3$ satisfy the quaternion relations 
\begin{equation}
	I_i\circ I_j=I_k \text{\ \ and\ \ }I_i\circ I_j=-I_j\circ I_i \label{quaternionrelations}\end{equation}
where $(i,j,k)$ are cyclic permutation of $(1,2,3)$. If $I_i$ are all complex structures then $(M^{4n},I_1,I_2,I_3)$ is called hypercomplex. In fact, the quaternion relations imply that if any two of the almost complex structures is integrable then so is the third.

A metric $g$ is hyper-Hermitian if it is Hermitian for each $I_i$. We also have a triple of $2$-forms $\om_i$ associated to each complex structure $I_i$ by $g(I_i(\cdot),\cdot)=\om_i(\cdot,\cdot)$. $(M^{4n},g,I_1,I_3,I_3)$ is called hyperK\"ahler if in addition each $\om_i$ is closed. This is equivalent to saying that the holonomy  group of $g$ is contained in $Sp(n)$. In particular, $g$ is then Ricci-flat.

If a Riemannian manifold $(M^{4n},g)$ admits \textit{locally} a triple of compatible almost complex structure satisfying (\ref{quaternionrelations}) then it is called almost quaternion Hermitian. Note that although the almost complex structures are not required to be globally well-defined the $4$-form $\Om$ locally defined by the expression
\[\Om:=\frac{1}{2}(\om_1 \w \om_1 + \om_2 \w \om_2 + \om_3 \w \om_3)\]
is in fact globally well-defined. An equivalent way of phrasing this is saying that although the almost complex structures $I_i$ can only be chosen locally the unit $2$-sphere in the rank $3$ subbundle $\mathcal{G}=\langle I_1,I_2,I_3\rangle \subset\text{End}(TM)$ is globally well-defined.

For $n\geq 2$, $(M^{4n},g,\Om)$ is called a quaternion-K\"ahler manifold if $\Om$ is parallel with respect to the Levi-Civita connection $\nabla^g$ of $g$. Note that QK manifolds are not hypercomplex in general, consider for instance $\mathbb{H}\mathbb{P}^n$. 
It is well-known that they are Einstein manifolds i.e. $Ric(g)=\lambda\cdot g$ \cite{Salamon1982}. If $\lambda =0$ then $M$ is locally hyperK\"ahler so this is excluded in the definition of QK manifolds. If $\lambda>0$ then the only known examples are the Wolf symmetric spaces \cite{Wolf1965}. By contrast, as demonstrated by LeBrun in \cite{LeBrunInfiniteQK}, the $\lambda<0$ case can even occur in infinite families and this will be the case of interest in this paper.
In \cite{Swann1991} Swann shows that 
if $n>2$ then $\nabla^g \Om=0$ if and only if $d \Om=0$. If $n=2$ then in addition to $d\Om=0$, One also requires that the algebraic ideal generated by $\langle\om_1,\om_2,\om_3 \rangle$ is a differential ideal i.e. the algebraic ideal is closed under differentiation.

In \cite{Salamon2001} Salamon gave the first example of a compact quaternion Hermitian $8$-manifold with $d\Om=0$ but whose algebraic ideal $\langle\om_1,\om_2,\om_3 \rangle$ is not differential. This confirms that indeed closedness of $\Om$ is not sufficient for $n=2$.
More examples were constructed in \cite{ContiMadsen2015}. None of these examples are not Einstein however. When $n=1$ one defines QK manifolds as $4$-manifolds which are Einstein and self-dual i.e. the anti-self-dual part of the Weyl curvature vanishes e.g. $S^4$ and $\C \mathbb{P}^2$ \cite{AtiyahHitchinSinger78}.

The holonomy group of a QK manifold $(M^{4n},g,\Om)$ is contained in $Sp(n)Sp(1)$. Since $Sp(n)Sp(1) \subset SO(4n)$ and $\mathfrak{so}(4n)\cong \Lm^2$ we have the splitting (\ref{spnsp1splitting}) of the space of $2$-forms as $Sp(n)Sp(1)$-modules. As already mentioned above if $M$ is a HK manifold then the subbundle $\mathfrak{sp}(1)\cong \langle \om_1,\om_2,\om_3\rangle$ can be trivialised by parallel sections.
For our purpose we will need a more concrete description of the subbundle determined by $\mathfrak{sp}(n)$. 
Locally with respect to an almost complex structure $J\in \mathcal{G}$ we have the decomposition
\[\Lm^2\otimes \C=\Lm^{2,0}_{J}\oplus \Lm^{1,1}_{J}\oplus\Lm^{0,2}_{J}\]
and hence we have $\mathfrak{sp}(n)\subset \mathfrak{u}(2n) \cong [\Lm^{1,1}_J]$, where by definition $[\Lm^{1,1}_J]\otimes \C = \Lm^{1,1}_J$ \cite{Salamon1989}. Since this holds for any $J\in\mathcal{G}$ 
together with an argument using Schur's lemma it follows that 
\begin{equation}
	\mathfrak{sp}(n) \cong \bigcap_{J\in \mathcal{G}}[ \Lm^{1,1}_{J}].\label{spnalgebra}
\end{equation}
So the subbundle determined by $\mathfrak{sp}(n)$ consists of $2$-forms which are of type $(1,1)$ with respect to any almost complex structure $J\in \mathcal{G}$ cf. \cite[Proposition 1]{SalamonYM88}. With these facts in mind we now proceed to prove Theorems \ref{maintheorem}, \ref{hypercomplextheorem} and \ref{extendedhypercomplextheorem}.

\section{Proofs of Theorems \ref{maintheorem}, \ref{hypercomplextheorem} and \ref{extendedhypercomplextheorem}}\label{proofs}
\begin{proof}[Proof of Theorem \ref{maintheorem}]
Let us first show that (\ref{qkmetric}) is a QK metric on $N^{4n+4}$. 
In view of the aforementioned result of Swann, we need to show that $\Om_{N}$ is closed and
for the $n=1$ case we also need to verify that the algebraic ideal generated by $\langle \om_1, \om_2, \om_3\rangle$
is closed under differentiation. A straightforward calculation using (\ref{connection1})-(\ref{connection3}) shows that 
\begin{equation}\label{qkequ}
\begin{cases}
\begin{matrix}
	d\om_1 &=& 0 &+& (-4 e^{2t} \eta)\w \om_2 &+& (4 e^{2t}\xi )\w \om_3,\\
	d\om_2 &=& (4 e^{2t} \eta)\w \om_1 &+& 0  &+& (4 e^{2t}\al )\w \om_3,\\
	d\om_3 &=& (-4 e^{2t} \xi)\w \om_1 &+& (-4 e^{2t}\al )\w \om_2 &+& 0,
\end{matrix}\end{cases}\end{equation}
which confirms the latter. One also easily sees from (\ref{qkequ}) that \[d\Om_{N}=\sum_{i=1}^{3}\om_i \w d\om_i=0.\]
This concludes the proof that (\ref{qkmetric}) is a QK Einstein metric for all $n$. The fact that the almost complex structures $I_i$ are all integrable will follow from Theorem \ref{hypercomplextheorem} which we prove below.

We now show that (\ref{kahlereinstein}) is a K\"ahler metric on $P^{4n+2}$. It is clear from the structure equations that $d\om_P=0$. So it suffices to show that the associated almost complex structure is integrable. Observe that the form $dt+2ie^{2t}\al$ is of  type $(1,0)$ and $\sigma_2 + i\sigma_3$ is of type $(2,0)$ with respect to the associated almost complex structure. We have that
\[d((\frac{1}{2e^{2t}} dt+i\al) \w (\sigma_2 + i\sigma_3)^n)=i \sigma_1 \w (\sigma_2+i\sigma_3)^n=0\]
i.e. $d(\Gamma(\Lm^{2n+1,0}))\subset \Gamma(\Lm^{2n+1,1})$ and hence the almost complex structure is integrable. In fact the above calculation shows that $c_1(P)=0.$

Since QK manifolds are Einstein to complete the proof of the theorem it only remains to show that $g_P$ and $g_L$ are Einstein metrics as well.

Consider more generally a metric of the form \begin{equation}g=dt^2+p(t)^2g_M+q(t)^2\al^2+r(t)^2\xi^2+s(t)^2\eta^2\end{equation}
on $N^{4n+4}$, then the Einstein equation $Ric(g)=\lambda \cdot g$ is given by the system:
\begin{gather}
	\frac{q''}{q}+\frac{r''}{r}+\frac{s''}{s}+4n\cdot \frac{p''}{p}=-\lambda,\label{firstequ}\\
	qq''+qq'(\frac{r'}{r}+\frac{s'}{s})+4n\cdot \frac{p'qq'}{p}-n\cdot \frac{q^4}{p^4}=-\lambda\cdot q^2,\\
	rr''+rr'(\frac{q'}{q}+\frac{s'}{s})+4n\cdot \frac{p'rr'}{p}-n\cdot \frac{r^4}{p^4}=-\lambda\cdot r^2,\\
	ss''+ss'(\frac{r'}{r}+\frac{q'}{q})+4n\cdot \frac{p'ss'}{p}-n\cdot \frac{s^4}{p^4}=-\lambda\cdot s^2,\\
	pp''+pp'(\frac{q'}{q}+\frac{r'}{r}+\frac{s'}{s})+(4n-1)\cdot(p')^2+\frac{1}{2p^2}(q^2+r^2+s^2)=-\lambda\cdot p^2.\label{lastequ}
\end{gather}
For the metric ansatz
 \begin{equation}g=dt^2+p(t)^2g_M+q(t)^2\al^2+r(t)^2\xi^2\end{equation}
 on $L^{4n+3}$, the corresponding Einstein equation is obtained simply by eliminating the terms involving $s(t)$ in (\ref{firstequ})-(\ref{lastequ}). Likewise, for the metric ansatz
 \begin{equation}g=dt^2+p(t)^2g_M+q(t)^2\al^2 \label{kahlereinsteinansatz}\end{equation}
 on $P^{4n+2}$, one eliminates terms involving $r(t)$ and $s(t)$ in (\ref{firstequ})-(\ref{lastequ}) to get the corresponding Einstein equation and finally for the ansatz
\begin{equation} g=dt^2+p(t)^2g_M \end{equation}
on $Q^{4n+1}$ the whole system reduces to the well-known warped product Einstein equation
\begin{equation} (p')^2+\frac{\lambda}{4n}\cdot p^2=0. \end{equation}
One verifies directly that $p(t)=a e^{bt}$ and $q(t)=r(t)=s(t)=2a^2b e^{2bt}$ solves the Einstein equations in all the cases. By rescaling the metric we can always set $b=1$. The factor $a$ can be interpreted geometrically as follows. If  $a$ is an integer we can instead consider the $\mathbb{T}^3$ bundle $M_{a\al,a\xi,a \eta}^{4n+3}$ on $M^{4n}$ determined by $[a\sigma_1], [-a\sigma_2], [-a\sigma_3]$. Thus, the parameter $a$ essentially corresponds to pullbacking the metric on covers of the $\mathbb{T}^3$ fibres. As a consequence of this observation we can in fact relax the hypothesis of the theorem to simply requiring that $[\om_i] \in H^{2}(M, 2\pi\mathbb{Q}).$
Since the metrics are well-defined for all $t\in \R$ it follows that they are complete, provided of course that $g_M$ is. 
\end{proof}
In view of the Ambrose-Singer theorem cf. \cite[Theorem 8.1]{KobayashiNomizu1}, one can verify by computing the rank of the curvature operator, say using \textsc{Maple}, that the Einstein metrics $g_Q$, $g_P$, $g_L$ and $g_N$ have (restricted) holonomy groups equal to $SO(4n+1)$, $U(2n+1)$, $SO(4n+3)$ and $Sp(n)Sp(1)$ respectively. 
It is also clear that each of the projection maps
\[(N^{4n+4},g_N)\to (L^{4n+3},g_L) \to (P^{4n+2},g_P)\to (Q^{4n+1},g_Q)\to (M^{4n},g_M)\]
is a Riemannian submersion. Moreover, the $\mathbb{T}^2$ bundle $N^{4n+4}\to P^{4n+2}$ is a holomorphic fibration with respect to the complex structure $I_1$ on $N^{4n+4}$.

Geometrically, one can interpret the metrics of Theorem \ref{maintheorem} as twisting the HK metric $g_M$ with the hyperbolic metrics. For instance, the metric $g_L$ restricted to a fibre $\R_t \times \mathbb{T}^2$ (over a point in $M$) can be expressed as 
\begin{equation}
dt^2+4e^{4t}(dy_1^2+dy_2^2).\label{hyperbolic3}
\end{equation}
for local coordinates $y_1,y_2$ on the $\mathbb{T}^2$. Thus, $g_L$ can be thought of as a twisted sum of $g_M$ with (\ref{hyperbolic3}). Likewise the same is applies to $g_P$ and $g_N$.

It is worth emphasising that the Einstein metrics arising from Theorem \ref{maintheorem} are not homogeneous unless $M^{4n}$ is flat, in which case $Q,P,L$ and $N$ correspond to certain solvmanifolds cf. \cite{Gibbons01}.
More generally, homogeneous Einstein metrics on solvmanifolds have been investigated by many authors cf. \cite{Alekseevskii1975, Heber1998, Lauret2002, Will2003, Wolter1991} and more recently it was showed in \cite{BohmLafuente2021} that the Alekseevskii conjecture is true i.e any homogeneous Einstein manifold with negative scalar curvature is diffeomorphic to a Euclidean space. So the examples arising from Theorem \ref{maintheorem} differ from those cases.
\begin{Rem}
In the case when we considered ansatz (\ref{kahlereinsteinansatz}) on $P^{4n+2}$, we define a non-degenerate $2$-form (or, equivalently, an almost complex structure) by
\begin{equation}
\om_P=qdt \w \al+p^2\sigma_1.
\end{equation}
Then the $(2n+1,0)$-form 
\begin{equation}
\Om_P:=(q^{-1}dt+i\al)\w(\sigma_2+i\sigma_3)^n 
\end{equation}
is automatically closed and hence holomorphic. Thus, $P^{4n+2}$ is always a Hermitian manifold.
The requirement that $\om_P$ is closed becomes 
\[q=2pp'.\]
Assuming the latter holds, $P$ is then K\"ahler and we can compute its Ricci form as
\begin{align}
Ric(\om_P)=&\ i \partial \bar{\partial}(\log(\|\Om_P\|^2_{g_P}))\nonumber\\
=&\ dd^c(\log(p^{2n+1}p')).\label{riccikahlerpde}
\end{align}
With $p=e^t$ we get
\[Ric(\om_P)=-4(n+1)\cdot\om_P\]
confirming that it is Einstein. We can also easily deduce the K\"ahler potential from
\[\om_P=-\frac{1}{2}dd^c(t).\]
In this particular situation solving the Einstein equation was rather simple owing to the K\"ahler structure but in general, one has to solve the more complicated second order system (\ref{firstequ})-(\ref{lastequ}).
\end{Rem}
We now move to the proof of Theorem \ref{hypercomplextheorem}.
\begin{proof}[Proof of Theorem \ref{hypercomplextheorem}]
We want to show that the almost complex structures $\check{I}_1,\check{I}_2,\check{I}_3$ associated to $\check{\om}_1, \check{\om}_2, \check{\om}_3$ are all integrable. Let us first consider  the almost complex structure $\check{I}_1$. The complex form 
\begin{equation}
	\check{\Upsilon}_1:=(\sigma_2 + i \sigma_3)^{n}\w (\xi + i \eta)\w (\nu + i \al)
\end{equation}
is of type $(2n+2,0)$ with respect to $\check{I}_1$. The Nijenhuis tensor of $\check{I}_1$ is determined by the $(2n-1,2)$ component of the exterior differential of $\check{\Upsilon}_1$. Computing we get 
\begin{align*}
d\check{\Upsilon}_1=&\ -(\sigma_2+i\sigma_3)^{n+1}\w (\nu + i \al)-(\sigma_2+i\sigma_3)^n \w (\xi + i \eta)\w (\gamma+i \sigma_1)\\
=& \ -(\sigma_2+i\sigma_3)^n \w (\xi + i \eta)\w \gamma\\
=& \ 0
\end{align*}
where we used that $(\sigma_2+i\sigma_3)^{n+1}=0\in \Lm^{2n+2,0}_{I_1}(M^{4n})$ and $(\sigma_2+i\sigma_3)^{n}\w \sigma_1\in \Lm^{2n+1,1}_{I_1}(M^{4n})$ for the second equality and for the last equality we used the fact that
$\gamma \in \mathfrak{sp}(n)$, see (\ref{spnalgebra}), so $(\sigma_2+i\sigma_3)^{n}\w \gamma=0\in \Lm^{2n+1,1}_{I_1}(M^{4n})$.
To complete the proof it suffices to apply the analogous argument for $\check{I}_2$ and $\check{I}_3$ with 
\begin{equation}
	\check{\Upsilon}_2:=(\sigma_3 + i \sigma_1)^{n}\w (\xi + i \nu)\w (\al + i \eta)
\end{equation}
and
\begin{equation}
	\check{\Upsilon}_3:=(\sigma_1 + i \sigma_2)^{n}\w (\xi + i \al)\w (\eta + i \nu).
\end{equation}
\end{proof}
Taking $$\nu=\frac{1}{2}e^{-2t}dt$$ and $M^{4n+4}_{\al,\xi,\eta,\nu}=\R_t \times M^{4n+4}_{\al,\xi,\eta}$ (by lifting to the universal cover of trivial $S^1_{\nu}$ bundle) we see that $\check{I}_1,\check{I}_2,\check{I}_3$ agree with $I_1,I_2,I_3$: on the fibres we have
\begin{align*}
I_1(\xi) &=-\eta, & I_1(\nu) &=-\al, \\ 
I_2(\xi) &=-\nu, & I_2(\al) &=-\eta, \\
I_3(\xi) &=-\al, & I_3(\eta) &=-\nu, 
\end{align*}
and it is clear that they agree on the horizontal space. Thus, $\langle I_1,I_2,I_3\rangle$ define a hypercomplex structure on $N^{4n+4}$.
 \begin{Rem}
 The proof of Theorem \ref{hypercomplextheorem} in fact shows that $\check{\Upsilon}_i$ are holomorphic $(2n+2,0)$-forms for $\check{I}_i$ and hence the canonical bundles are trivial. However, if either one of $[\sigma_i],[\nu]$ is non-trivial in cohomology then $M^{4n+4}_{\al,\xi,\eta,\nu}$ cannot even be K\"ahler. Say $[\sigma_1]$ is non-trivial then a simple calculation using the Gysin sequence shows that the $S^1$ fibres of $M_{\al}$ are trivial in homology cf. \cite{Goldstein2004}. It follows that a $\mathbb{T}^2$ fibre in $M_{\al,\xi,\eta,\nu}$ determined by $[\sigma_1]$ and $[\sigma_2]$ is a homologically trivial elliptic curve with respect to $I_3$. Now any $I_3$-compatible K\"ahler form would integrate to zero on this $\mathbb{T}^2$ but on the other hand this corresponds to its volume which is a contradiction. This argument can be applied to any of the complex structures by choosing suitable $\mathbb{T}^2$.
 In particular, we deduce that $M^{4n+4}_{\al,\xi,\eta,\nu}$ is never K\"ahler if $M^{4n}$ is compact.
 \end{Rem}
A slight modification to the argument in the proof of Theorem \ref{hypercomplextheorem} leads to the following different class of hypercomplex manifolds: 
\begin{Th}\label{extendedhypercomplextheorem}
	Let $(M^{4n},\sigma_1,\sigma_2,\sigma_3)$ be a hyperK\"ahler manifold such that there exists a quadruple of closed $2$-forms $\gamma_i \in \mathfrak{sp}(n)$ with $[\gamma_i]\in H^{2}(M,2\pi\mathbb{Z})$, which without loss of generality one can assume are linearly independent or zero. Then 
	\begin{gather}
		\check{g} := g_M+\nu_1^2+\nu_2^2+\nu_3^2+\nu_4^2,\\
		\begin{cases}
			\check{\om}_1 :=\sigma_1+\nu_1 \wedge \nu_2 + \nu_3 \wedge  \nu_4, \\
			\check{\om}_2	:=\sigma_2+\nu_1 \wedge \nu_3 + \nu_4 \wedge  \nu_2, \\
			\check{\om}_3 :=\sigma_3+\nu_1 \wedge \nu_4 + \nu_2 \wedge \nu_3,
		\end{cases}
	\end{gather}
	where $\nu_i$ are connection $1$-forms such that $d\nu_i=\gamma_i$, define a hyper-Hermitian structure on the total space of the $\mathbb{T}^4$ bundle $M^{4n+4}_{\nu_1,\nu_2,\nu_3,\nu_4}$ determined by $[\gamma_i]$.
\end{Th}
\begin{proof}[Proof of Theorem \ref{extendedhypercomplextheorem}]
The proof consists of just repeating the argument in the proof of Theorem \ref{hypercomplextheorem} but replacing $\al$, $\xi$, $\eta$ and $\nu$ by $\nu_1$, $\nu_2$, $\nu_3$ and $\nu_4$ respectively and using (\ref{spnalgebra}). 
\end{proof}
Note that by contrast to Theorem \ref{maintheorem} and \ref{hypercomplextheorem}, in Theorem \ref{extendedhypercomplextheorem} we do not require that the hyperK\"ahler triple define integral cohomology classes. Examples of hypercomplex structures of this type have been studied by Fino and Dotti in \cite{Dotti2000} when $M=\mathbb{T}^4$ (see section \ref{T4example} below).
\begin{Rem}
The fact that $d\nu_i$ are curvature forms of type $(1,1)$ on $M^{4n}$ is equivalent to saying that the connection forms $\nu_i$ are (abelian) Hermitian Yang-Mills connections, and this is true with respect to any of the hyperK\"ahler triple. In the terminology of \cite{SalamonYM88} these connections are called self-dual instantons. Twistor theory asserts that self-dual instantons on HK (and QK manifolds) lift to the twistor space, see for instance section \ref{twistorsection}, to define holomorphic vector bundles cf. \cite{AtiyahHitchinSinger78, HKLR, Salamon1982}. By contrast, the connection forms of Theorem \ref{maintheorem} are called anti-self-dual instantons since the curvature forms are of type $(2,0)+(0,2)$ instead. In both cases however the connections are critical points of the Yang-Mills functional
\[\text{YM}(A):=\int_M\|F_A\|^2_{g_M}\vol_M,\]
where $A$ is a connection form and $F_A$ its curvature i.e. they satisfy $d*_MF_A=0$. The latter is readily seen from the fact that $*_{M}\sigma_i=c_0 \sigma_i^{2n-1}$ and $*_{M}\gamma_i=c_1 \gamma_i \w \sigma_1^{2n-2}$ for some constants $c_0,c_1.$
\end{Rem}
It is easy to see that one can replace the hypothesis that $M^{4n}$ is hyperK\"ahler with just hypercomplex and still conclude that $M^{4n+4}_{\nu_1,\nu_2,\nu_3,\nu_4}$ is hypercomplex. The only thing we lose in doing so is that there is no longer an $Sp(n+1)$-structure with trivial canonical bundle but instead just a $GL(n+1,\mathbb{H})$-structure (i.e. quaternionic structure) on the total space $M^{4n+4}_{\nu_1,\nu_2,\nu_3,\nu_4}$. 

 \section{Examples}\label{examples}
 
In this section we construct various examples that arise from Theorems \ref{maintheorem}, \ref{hypercomplextheorem} and \ref{extendedhypercomplextheorem}. We consider primarily the situation when $M$ is $4$-dimensional as this is the most well-understood one. 

We begin by considering the case when $M$ is compact. Kodaira's classification states that the only compact complex surfaces with trivial canonical bundle are $K3$ surfaces and $\mathbb{T}^4$ cf. \cite{Barth}. In the former case one needs to study the $K3$ lattice to find integral classes. We shall instead consider the latter case as it is more explicit. 
\subsection{Examples with $M=\mathbb{T}^4$}\label{T4example} Denote by $(x_1,x_2,x_3,x_4)$ the local coordinates on $\mathbb{T}^4$ and consider the standard HK triple:
 \begin{align*}
 	\sigma_1=dx_{12}+dx_{34},\\
 	\sigma_2=dx_{13}+dx_{42},\\
 	\sigma_3=dx_{14}+dx_{23},
 \end{align*}
which after suitable rescaling we may assume define integral cohomology classes. The manifolds $M^{5}_{\al}$, $M^{6}_{\al,\xi}$ and $M^{7}_{\al,\xi,\eta}$ then correspond to nilmanifolds with nilpotent Lie algebras
\begin{gather}
	(0,0,0,0,12+34),\\
	(0,0,0,0,12+34,13+42),\\
	(0,0,0,0,12+34,13+42,14+23),\label{hyperbolicAlg}
\end{gather}
in Salamon's notation cf. \cite{Salamoncomplexstructures}; in this convention $(0,0,0,0,12+23)$ represents the Lie algebra with an invariant coframing $e^i$ such that $de^1=de^2=de^3=de^4=0$ and $de^{5}=e^{12}+e^{34}.$

The Lie algebra (\ref{hyperbolicAlg}) corresponds to the quaternion Heisenberg group, which is the nilpotent part in the Iwasawa decomposition of the isometry group of the quaternionic hyperbolic space \cite{Dotti2000}. It is thus not surprising that the universal cover of $N^{8}$ corresponds to the quaternion hyperbolic space $\frac{Sp(2,1)}{Sp(2)Sp(1)}$ as was shown in \cite{Gibbons01}. 
More generally, we have that $\tilde{N}^{4n+4}=\frac{Sp(n+1,1)}{Sp(n+1)Sp(1)}$ when $M=\mathbb{T}^{4n}$.
 
Recall from Theorem \ref{hypercomplextheorem} that we need $\nu \in \mathfrak{sp}(n)$ to construct hypercomplex structures. In dimension $4$ this means that $\nu \in \Lm^{2}_-(M)$ i.e. $\nu$ is an anti-self-dual $2$-form.
So we can construct distinct hypercomplex structures by choosing $\nu$ in the $\mathbb{Z}$-module generated by 
$$\langle dx_{12}-dx_{34},\ dx_{13}-dx_{42},\ dx_{14}-dx_{23} \rangle.$$
For instance, if $\nu=dx_{12}-dx_{34}$ then $M^{8}_{\al,\xi,\eta,\nu}$ is the nilmanifold with Lie algebra
\[(0,0,0,0,12,34,13+42,14+23).\]
If one applies Theorem \ref{extendedhypercomplextheorem} with all $\nu_i=0$ then we just get the product hypercomplex structure on $M \times \mathbb{T}^4$.
If instead one sets 
$$\nu_1=dx_{12}-dx_{34},\ \ \nu_2= dx_{13}-dx_{42},\ \ \nu_3= dx_{14}-dx_{23},\ \ \nu_4=0$$ 
then we recover the abelian hypercomplex structure studied in \cite{Dotti2000} on
\[(0,0,0,0,0,12-34,13-42,14-23).\]
Note that although we get an isomorphic Lie algebra from Theorem \ref{hypercomplextheorem} with $\nu=0$, the hypercomplex structures are different. These examples all generalise naturally to the case when $M=\mathbb{T}^{4n}$, of course the space $\mathfrak{sp}(n)$ is then much larger so one can construct more examples. Next we consider the case when $M^4$ is non-compact.

\subsection{Examples from the Gibbons-Hawking ansatz}\label{ghansatz}
Consider an open set $B \subset \R^3$ with coordinates $\mathbf{u}=(u_1,u_2,u_3)$ and endowed with the flat Euclidean structure. Suppose that $V:B\to \R^+$ is a positive harmonic function such that $[-*dV]\in H^{2}(B,2\pi \mathbb{Z})$, where $*$ is the Hodge star operator on $B$, and denote by $M^4$ the associated $S^1$-bundle. 
The Gibbons-Hawking ansatz then asserts that 
\begin{align*}
	\sigma_1 &= \theta \w du_1 + V du_{23},\\
	\sigma_2 &=\theta \w du_2 + V du_{31},\\
	\sigma_3 &=\theta \w du_3 + V du_{12}
\end{align*}
define a HK triple on the $M^4$, where $\theta$ is a connection $1$-form such that $d\theta=-*dV$ \cite{GibbonsHawking}. If we take
\[V=c+\sum_{i=1}^k\frac{1}{2|\mathbf{u}-\mathbf{p}_i|},\]
where $\mathbf{p}_i$ are $k$ distinct points in $\R^3$ and $c\geq 0$ then we get complete HK metrics called multi-Taub-NUT if $c\neq 0$ and multi-Eguchi-Hanson if $c=0$. Euclidean $\R^4$ is a special case of the latter when $k=1$. Although the harmonic function $V$ is singular at $\mathbf{p}_i$ the HK metric $g_M$ extends smoothly to a fibration over all of $\R^3$ with the $S^1$ fibres collapsing to points over $\mathbf{p}_i$. 

Suppose now that $\mathbf{p}_i$ are chosen to lie on the line $u_2=u_3=0$ with $p_i<p_{i+1}$. It is well-known that $H_2(M,\mathbb{Z})$ is generated by complex curves of self-intersection $-2$. These curves can be constructed by lifting the line segment connecting $\mathbf{p}_i$ to $\mathbf{p}_{i+1}$ in $B$ to $M^4$. For instance if $k=2$ and $c=0$ then we get the usual Eguchi-Hanson metric on $T^*\C \mathbb{P}^1 \cong \mathcal{O}(-2).$
The condition that $[\sigma_i]$ define integral cohomology class becomes equivalent to asking that their integral over each of these curves is an integer. It is clear by construction that $[\sigma_2]$ and $[\sigma_3]$ always integrate to zero while integrating $[\sigma_1]$ gives $p_{i+1}-p_i$. So it suffices to choose $p_i \in \mathbb{Z}$ to obtain infinitely many distinct HK $4$-manifolds satisfying the hypothesis of Theorems \ref{maintheorem} and \ref{hypercomplextheorem}.
\begin{Rem}
Note that although the QK metric is completely determined by the harmonic function $V$, it is rather complicated to write down an explicit expression for $g_\Om$ using the Gibbons-Hawking coordinates since the connection forms $\al,\xi$ and $\eta$ are implicitly determined by $V$ i.e. it is hard to find primitive $1$-forms for $\sigma_i$ in the Gibbons-Hawking coordinates.
\end{Rem}
\begin{Rem}
	In \cite{LeBrunInfiniteQK} LeBrun showed that there exist small deformations of the quaternionic hyperbolic space yielding complete QK metrics on $\R^8$. On the other hand we know that the Taub-NUT metric can be viewed as a deformation of the flat metric on $\R^4$, thus it seems plausible that when $M^4$ is taken to be the Taub-NUT space then $N^8$ arising from Theorem \ref{maintheorem} is a special case of those of LeBrun.
\end{Rem}
There is also a higher dimensional version of the Gibbons-Hawking ansatz for constructing toric HK $4n$-manifolds cf. \cite{BielawskiDancer2000,PetersenPoon1988} that one can use to find suitable higher dimensional HK manifolds as above.
If one is only interested in local examples however then one can simply take $M^{4n}$ to be an open ball in any HK manifold since locally we always have $[\sigma_i]=0$.
Note that the integrality condition on the HK triple can be circumvented if one chooses instead to work on principal $\R$-bundles rather than $S^1$-bundles on $M^{4n}$ cf. \cite[Sec. 3(E)]{HKLR}. 

At this point one might ask if given a Killing vector field on $M$, can one lift it to a Killing vector field on the total spaces $Q,P,L,N$ of Theorem \ref{maintheorem}?
To address this question in the next section we study the infinitesimal symmetries of the Einstein manifolds constructed in Theorem \ref{maintheorem}.

\section{Symmetries and moment maps}\label{symmetriessection}
Analogous to the well-known Marsden-Weinstein symplectic reduction and its hyperK\"ahler extension described in \cite{HKLR},  Galicki-Lawson developed the theory of QK reduction in \cite{Galicki1988}. This gives a way of taking quotients of QK manifolds to produce new QK manifolds. In this section we describe how one can apply this construction to the QK manifolds $N^{4n+4}$ of Theorem \ref{maintheorem}.
\subsection{The QK reduction} 
We recall briefly the QK reduction. Let $(N,g_N,\Om_{N})$ be a general QK manifold and let $\tilde{X}$ be a Killing vector field such that $\mathcal{L}_{\tilde{X}}\Om_{N}=0$. Then Galicki-Lawson show that there exists a unique map $f_{\tilde{X}}:N \to \mathcal{G}$ such that 
\[df_{\tilde{X}}=\tilde{X} \ip \Om_{N} \]
called the QK moment map for the action generated by $\tilde{X}$. This applies more generally when we have a $G$-action preserving $g_N$ and $\Om_{N}$ but we shall restrict to the case when $G$ is $1$-dimensional i.e. generated by a single vector field $\tilde{X}$. 

If $G$ acts freely and properly on $f_{\tilde{X}}^{-1}(0)$ then the quotient space
\[N_{red}:={f_{\tilde{X}}^{-1}(0)}/{G}\]
is a QK manifold cf. \cite{Galicki1988}. 
If the $G$-action is only locally free then $N_{red}$ is instead an orbifold. 
Note that in contrast to the symplectic reduction, however, one can only take the quotient at the zero section in $\mathcal{G}$. This is essentially due to the fact that there are no ``constants'' in $\mathcal{G}$ to add, however as we shall see below in our case one can consider the level sets at $\om_i$. 

\subsection{Reduction using tri-holomorphic Killing vector fields}
Let $X$ be a tri-holomorphic Killing vector field on $M^{4n}$ i.e. $\mathcal{L}_X \sigma_i=0$ for $i=1,2,3$. Working locally there exists $1$-forms $\kappa_i$ such that 
\begin{equation}
	\sigma_i=d\kappa_i,\label{exactsymplec}
\end{equation}
and by choosing local coordinates $y_1,y_2,y_3$ on the $\mathbb{T}^3$ fibres of ${M}^{4n+3}_{\al,\xi,\eta}$ we can write
\begin{align}
	\al &=\ dy_1 + \kappa_1,\label{list1}\\
	\xi &=\ dy_2 - \kappa_2,\\
	\eta &=\ dy_3 - \kappa_3.
\end{align}
By Poincar\'e Lemma we can choose $\kappa_i$ so that $\mathcal{L}_X \kappa_i=0$ and hence the natural lift $\tilde{X}$ of $X$ (satisfying $dy_i(\tilde{X})=0$) to $N$ preserves $\al, \xi$ and $\eta$. 
The functions
\begin{align}
	\mu_\al &:=\ -\tilde{X} \ip \al,\\
	\mu_\xi &:=\ \tilde{X} \ip \xi,\\
	\mu_\eta &:=\ \tilde{X} \ip \eta,\label{list6}
\end{align}
then define the HK moment maps for the action generated by $X$ on $M$ since from Cartan's formula we have
\[d\mu_\al=-d(\tilde{X} \ip \al)=-\mathcal{L}_{\tilde{X}}\al+X\ip d\al=X\ip \sigma_1,\]
and likewise for $\mu_\eta$ and $\mu_\eta$. It is clear by construction that $\tilde{X}$ is Killing and preserves $\Om_{N}$. We compute the QK moment map for the action generated by $\tilde{X}$ on $N$ as
\begin{equation}
	f_{\tilde{X}}=e^{2t}(\mu_\al \om_1+\mu_\xi \om_2 + \mu_\eta \om_3).
\end{equation}
Indeed one verifies using (\ref{qkequ}) that
\begin{align*}
\tilde{X}\ip \Om_{N} =\ &(e^{2t}d\mu_\al+4e^{4t}\mu_\xi \eta-4e^{4t}\mu_\eta \xi +2e^{2t}\mu_\al dt) \w \om_1\ + \\
&(e^{2t}d\mu_\xi-4e^{4t}\mu_\al \eta-4e^{4t}\mu_\eta \al +2e^{2t}\mu_\xi dt) \w \om_2\ +\\ &(e^{2t}d\mu_\eta+4e^{4t}\mu_\xi \al+4e^{4t}\mu_\al \xi +2e^{2t}\mu_\eta dt) \w \om_3 \\
=\ &df_{\tilde{X}}
\end{align*}
The zero set $f_{\tilde{X}}^{-1}(0) \subset N$ corresponds to the bundle $N\to M$ restricted to the set $(\mu_\al,\mu_\xi,\mu_\eta)^{-1}(0) \subset M$. Taking the quotient we see that $N_{red}$ is just the QK manifold we get from Theorem \ref{maintheorem} by taking the HK base to be $M_{red}$, the HK quotient of $M$ by the action generated by $X$ \cite{HKLR}. In other words, we have:
\begin{Prop}
The construction of Theorem \ref{maintheorem} commutes with the QK reduction of $N$ and  the HK reduction of $M$:
\begin{center}
	\begin{tikzcd}[column sep=tiny]
		N \ar[d,"T^3 \times \R",swap]\ar[rr,"/QK"]
		& [1.5em]
		&N_{red} \ar[d,"T^3 \times \R"]\\
		M \ar[rr,"/HK"]
		&[1.5em]
		& M_{red}
	\end{tikzcd}
\end{center}
\end{Prop}
As a trivial example consider the case when $M$ is a HK $4$-manifold. The HK moment map then corresponds to the functions $(x_1,x_2,x_3)$ in the notation of subsection \ref{ghansatz}. $M_{red}$ can be viewed as a discrete set of points in $M$ (say by fixing a point in each $S^1$ orbit) and $N_{red}$ is just the fibres over these points with the hyperbolic QK metric
\[g_{red}=dt^2+4e^{4t}(dy_1^2+dy_2^2+dy_3^2).\] 
\subsection{Reduction using permuting vector fields}\label{permutingsubsection}
Let us now assume that $X$ is a permuting Killing vector field on $M^{4n}$ i.e. it preserves one of K\"ahler forms and rotates the other two. Note that any Killing vector field on a HK manifold is either tri-holomorphic as above or is permuting. Without loss of generally we can then assume that $\mathcal{L}_X \sigma_1=0$ while $\mathcal{L}_X \sigma_2= -2\sigma_3$ and $\mathcal{L}_X \sigma_3= +2\sigma_2$.  We choose $\kappa_1$ as before and set
\begin{align}
	\kappa_2 &:= +\frac{1}{2}(X \ip \sigma_3),\label{perm1}\\
	\kappa_3 &:= -\frac{1}{2}(X \ip \sigma_2).\label{perm2}
\end{align}
This implies that (\ref{exactsymplec}) holds and that
\begin{equation}
\mathcal{L}_X \kappa_1=0, \ \ \ \ \mathcal{L}_X \kappa_2=-2\kappa_3,\ \ \ \ \mathcal{L}_X \kappa_3=+2\kappa_2.\label{permuting}\end{equation} 
We now define a lift of $X$ to $N$ by
\begin{equation}
	\tilde{X}=X-2\cdot (y_3 \partial_{y_2}-y_2 \partial_{y_3})
\end{equation}
and define $\al,\xi,\eta,\mu_\al,\mu_\xi,\mu_\eta$ by expressions (\ref{list1})-(\ref{list6}) as before.
An analogous computation shows that $\tilde{X}$ is Killing and $\mathcal{L}_{\tilde{X}} \Om_{N}=0$ (since by construction $\tilde{X}$ is also permuting on $\om_1,\om_2,\om_3$). We compute the QK moment map for $\tilde{X}$ as
\begin{equation}
	f_{\tilde{X}}=e^{2t}((\mu_\al-\frac{1}{2}e^{-2t}) \om_1+\mu_\xi \om_2 + \mu_\eta \om_3).
\end{equation}
Note that unlike before the functions $\mu_\xi,\mu_\eta$ cannot be interpreted as moment maps but $\mu_\al$ is still a K\"ahler moment map on $M^{4n}$ for $X$. We refer the reader to \cite{Hitchin2000} for several examples of permuting Killing vector fields.

We shall re-examine this construction in greater detail in section \ref{HKQKsection} below.
\subsection{Reduction using homothetic Killing vector fields}\label{homotheticsection}
Let us now assume that $X$ is a homothetic Killing vector field so that $\mathcal{L}_X \sigma_i=-2\sigma_i$ and define
\begin{equation}
	\kappa_i := -\frac{1}{2}(X \ip \sigma_i).\label{homotheticdef}
\end{equation}
This implies that (\ref{exactsymplec}) holds and that
\begin{equation}
	\mathcal{L}_X \kappa_i=-2\kappa_i.\label{homotheticequ}
\end{equation} 
We define a lift of $X$ to $N$ by
\begin{equation}
	\tilde{X}=X-2\cdot (y_1 \partial_{y_1}+y_2 \partial_{y_2}+y_3 \partial_{y_3})+\partial_t
\end{equation}
and define $\al,\xi,\eta,\mu_\al,\mu_\xi,\mu_\eta$ by expressions (\ref{list1})-(\ref{list6}).
A simple computation shows that $\mathcal{L}_{\tilde{X}} \Om_{N}=0$ and we find the QK moment map for $\tilde{X}$ is
\begin{equation}
	f_{\tilde{X}}=e^{2t}(\mu_\al \om_1+\mu_\xi \om_2 + \mu_\eta \om_3).
\end{equation}
\begin{Rem}\label{remark}
	It's worth pointing out that we only used that  $\kappa_i$ satisfy relations (\ref{exactsymplec}) and (\ref{homotheticequ}) to compute the moment map. In particular, one can choose $\kappa_i$ different from (\ref{homotheticdef}). This amounts to modifying $\kappa_i$ by a suitable closed $1$-form.
\end{Rem}
Lastly note that the vertical Killing vector fields $\partial_{y_1},\partial_{y_2},\partial_{y_3}$ always preserve $\Om_{N}$. The corresponding QK moment map for 
\begin{equation}
	Y=-a\partial_{y_1}+b\partial_{y_2}+c\partial_{y_3},\label{verticalvf}
\end{equation}
where $a,b,c \in \R$ is given by  
\begin{equation}
	f_{Y}=e^{2t}(a \om_1+b\om_2 + c \om_3).
\end{equation}
This moment map is nowhere vanishing so $N_{red}$ would be an empty set in this case. However, by taking linear combinations of $Y$ with $\tilde{X}$ in either of the previous $3$ cases we may add constants to the corresponding moment maps.
\begin{Rem}
	It is also easy to see that the vector fields $Y$ and $\tilde{X}$ (modified in the obvious way) are also Killing for the metrics $g_Q, g_P, g_L$ of Theorem \ref{maintheorem}. Moreover, we get a Hamiltonian action on $P^{4n+2}$ so one can investigate the K\"ahler reduction in this case.
\end{Rem}
\subsection{An explicit example}\label{explicithomotheticexample}
We illustrate the above reduction explicitly in the case when $M=\R^4$ endowed with a homothetic Killing vector field. 

Denoting by $(x_1,x_2,x_3,x_4)$ the coordinates on $\R^4$ we consider the standard HK triple as in section \ref{T4example}.
We take the homothetic Killing vector field
\[U=-x_1\partial_{x_1}-x_2\partial_{x_2}-x_3\partial_{x_3}-x_4\partial_{x_4},\]
which corresponds to scaling along the radial direction. From the definition of $\kappa_i$ we find that the connection forms are given by
\begin{gather*}
	\al=dy_1-\frac{1}{2}(x_2dx_1-x_1dx_2+x_4dx_3-x_3dx_4),\\
	\xi=dy_2+\frac{1}{2}(x_3dx_1-x_4dx_2-x_1dx_3+x_2dx_4),\\
	\eta=dy_3+\frac{1}{2}(x_4dx_1+x_3dx_2-x_2dx_3-x_1dx_4),
\end{gather*}
and the moment map becomes
\[f_{\tilde{U}}=2e^{2t}(y_1 \om_1-y_2\om_2-y_3 \om_3).\]
The reduction now yields a QK metric on $N_{red}\simeq\R^4$. In general, identifying the quotient metric explicitly is quite hard, however, in this rather simple situation we can write
\begin{equation}
	g_N\big|_{f^{-1}_{\tilde{U}}(0)}=y\big(dt -\frac{1}{2}d (\ln(y))\big)^2 + g_{red}
\end{equation}
where $y:=\cosh^{-1}(\sqrt{1+r^2e^{2t}})$ and $r:=\sqrt{x_1^2+x_2^2+x_3^2+x_4^2}$. The $1$-form $dt -\frac{1}{2}d (\ln(y))$ is just the canonical Riemannian connection for the action generated by $\tilde{U}$ i.e.
\[dt -\frac{1}{2}d (\ln(y))=\frac{g_\Om\big|_{f^{-1}(0)}(\tilde{U},\cdot)}{g_\Om\big|_{f^{-1}(0)}(\tilde{U},\tilde{U})}.\]
and the reduced QK metric $g_{red}$ can be explicitly expressed as
\begin{equation}
	g_{red}=dy^2+\sinh^2(y)\cosh^2(y)g_{S^3}.
\end{equation}
The above calculation is easily carried out using spherical  coordinates on $\R^4$. This quotient corresponds to the QK symmetric space $\frac{Sp(1,1)}{Sp(1)Sp(1)}$.
\begin{Rem}
	In view of remark \ref{remark} one can take say $\kappa_1$ to be $-x_2dx_1-x_4dx_3$ instead. Repeating the above calculation one can show that the quotient metric $g_{red}$ remains unchanged although the QK moment $f_{\tilde{U}}$ changes.
\end{Rem}
An example of a permuting Killing action in this situation is given by the diagonal action of $U(1)$ on $\C^2 \cong \R^4$ generated by the vector field
\[V=-x_2\partial_{x_1}+x_1\partial_{x_2}-x_4\partial_{x_3}+x_3\partial_{x_4}\]
and an example of a tri-holomorphic Killing vector field is given by
\[W=-x_2\partial_{x_1}+x_1\partial_{x_2}+x_4\partial_{x_3}-x_3\partial_{x_4}\]
corresponding to the action of the diagonal $U(1)$ in $SU(2)$ on $\C^2$. Taking linear combination of these vector fields one can investigate more general reductions of $N^8$. 
The QK moment map for the Killing vector field $\tilde{X}=u\tilde{U}+v\tilde{V}+w\tilde{W}+Y$, where $u,v,w \in \R$ is given by
\begin{align}
f_{\tilde{X}} =\ e^{2t}\big(&(a+2uy_1-\frac{v}{2}(r^2+e^{-2t})-\frac{w}{2}(x_1^2+x_2^2-x_3^2-x_4^2)) \om_1\ \label{generalQK}+\\ &(b-2uy_2-2vy_3-w(x_1x_4+x_2x_3))\om_2\ +\nonumber\\ 
&(c-2uy_3+2vy_2+w(x_1x_3-x_2x_4)) \om_3\big).\nonumber
\end{align}
Identifying the reduced metric explicitly appears to be rather hard in this more general situation.

\subsubsection{$3$-Sasakian manifolds}\label{3sasakian}
The above simple example can more generally be applied to any $3$-Sasakian manifold $\mathcal{S}^{4n-1}$. There are several equivalent definitions of $3$-Sasakian manifolds cf. \cite{GalickiSurvey3sasaki,GalickiSalamon1996}, the simplest to state is that they are Riemannian manifolds whose cone metric is hyperK\"ahler. More concretely, this means that they admit $3$ unit Killing vector fields, or equivalently using the metric a triple of $1$-forms $\gamma_i$, spanning a rank $3$ subbundle $\mathcal{F}$ of $T\mathcal{S}$ such that
$$d\gamma_i=2\gamma_{j}\w \gamma_{k}+2\check{\sigma}_i, $$
where $(i,j,k)$ are cyclic permutations of $(1,2,3)$ and 
with $\check{\sigma}_i$ defining a $Sp(n-1)$-structure on $\mathcal{F}^{\perp}.$ The hyperK\"ahler triple on the cone $M^{4n}:=\R^+_r \times \mathcal{S}$ are then given by
\[\sigma_i:=\frac{1}{2}d(r^2 \gamma_i).\]
The homothetic Killing vector field $X=-r\partial_r$ (i.e. the Euler vector field) satisfies the hypothesis of subsection \ref{homotheticsection}. Repeating the exact calculation as in case with $M=\R^4$ shows that 
\begin{equation}
	g_{red}=dy^2+\sinh^2(y) \cosh^2(y)(\gamma_1^2+\gamma_2^2+\gamma_3^2)+\sinh^2(y) g_{\mathcal{F}^\perp},\label{sasakinegative}\end{equation}
where $g_{\mathcal{F}^\perp}$ denotes the $3$-Sasakian metric restricted to $\mathcal{F}^\perp$. 

Note that this metric is complete only if $\mathcal{S}$ is the round sphere in which case $g_{red}$ corresponds to the quaternionic hyperbolic metric whereas the HK cone metric on $M$ corresponds to the flat Euclidean metric. Otherwise, for general $\mathcal{S}$ we have an isolated conical singularity at $y=0$.\\

More generally, if $X$ is a homothetic Killing vector field on HK $M^{4n}$ then following the same strategy as in the $M=\R^4$ case we can express $g_{red}$ as
\begin{gather}
g_{red}=dt^2+4e^{4t}(\sum_{i=1}^3\kappa_i^2)+e^{2t}g_M-(1+e^{2t}\|X\|_{g_M}^2)^{-1}(dt+e^{2t}g_M(X,\cdot))^2.\label{generalredmetrichomothetic}
\end{gather}

\subsubsection{Local examples from the Gibbons-Hawking ansatz}
If one chooses $V=u_1$ in the Gibbons-Hawking ansatz, see subsection \ref{ghansatz}, so that locally the connection form is $\theta=dy+u_3du_2$, where $y$ denotes the coordinate of the $S^1$ fibre then we still get a HK metric but it is incomplete. We leave it to the reader to verify that 
\[X=-\frac{2}{3}(2y\partial_y+u_1\partial_{u_1}+u_2\partial_{u_2}+u_3\partial_{u_3})\]
 is a homothetic Killing vector field satisfying the hypothesis of section \ref{homotheticsection}. Unlike in the previous case, however, this is not a gradient vector field. One can find many such examples that by choosing suitable $V$ and thus one can apply the above quotient construction. However, as $X$ is not a gradient vector field it appears to be much harder to find explicit local coordinates on the quotient and hence identify the reduced metric (\ref{generalredmetrichomothetic}). 

\subsection{QK products}
In general the product of two QK manifolds is not QK. In \cite{Swann1991} Swann gave a way of joining two QK manifolds $N_1, N_2$ of positive scalar curvature to construct a QK manifold of dimension $\dim(N_1)+\dim(N_2)+4$. We now describe a similar join construction which applies to the quotient manifolds of section \ref{homotheticsection}. 

Let $M_1, M_2$ be HK manifolds with homothetic Killing vector fields $X_1$ and $X_2$ respectively as in section \ref{homotheticsection}. The above quotient construction then produces two (negative scalar curvature) QK manifolds $M_{1red}$ and $M_{2red}$. Of course the product metric on $M_{1red} \times M_{2red}$ is not QK but one can still define a ``QK product'' of these spaces as follows. Observe that $M_1 \times M_2$ is a HK manifold with the standard product structure and the vector field $X_1+X_2$ defines a diagonal homothetic action. Applying the construction of section \ref{homotheticsection} to the data $(M_1 \times M_2, X_1+X_2)$ yields a QK manifold of dimension $\dim(M_{1red})+\dim(M_{2red})$ which we denote by $\mathcal{J}(M_1,M_2)$: the QK product of  $M_{1red}$ and $M_{2red}$. 

Here is an example. Let $M_{i}:=\R^+_{x_i}\times \mathcal{S}_i$ be the HK cone of the $3$-Sasakian manifold $\mathcal{S}_{i}$ for $i=1,2.$ Then product HK metric is given by
\begin{align}
g_{M_1}+g_{M_2} =& \ dx_1^2+dx_2^2 + x_1^2 g_{\mathcal{S}_1}+ x_2^2 g_{\mathcal{S}_2}\nonumber\\ 
=&\ dr^2 + r^{2}(d\theta^2+ \cos^2(\theta) g_{\mathcal{S}_1} + \sin^2(\theta) g_{\mathcal{S}_2}),
\end{align}
where we use polar coordinates $x_1=r\cos(\theta)$ and $x_2=r\sin(\theta)$. The metric 
 \begin{equation}
 	d\theta^2+ \cos^2(\theta) g_{\mathcal{S}_1} + \sin^2(\theta) g_{\mathcal{S}_2} \label{sinecone}
  \end{equation}
 is of course again a $3$-Sasakian metric. For instance, if one takes both $\mathcal{S}_{i}$ to be the standard $3$-sphere $S^3$ then (\ref{sinecone}) corresponds to the round metric on $S^7.$ The QK metric on $\mathcal{J}(M_1,M_2)$ can be explicitly worked out from (\ref{sasakinegative}).
 
 When $M_1=\R^{4n}$ and $M_2=\R^{4k}$ with the homothetic radial Killing vector fields as above we already saw that $M_{1red}=\frac{Sp(n,1)}{Sp(n)Sp(1)}$ and $M_{2red}=\frac{Sp(k,1)}{Sp(k)Sp(1)}$. Thus we deduce that the QK product is
 \begin{equation}
 	\mathcal{J}(\R^{4n},\R^{4k}) =\frac{Sp(n+k,1)}{Sp(n+k)Sp(1)}.
 \end{equation}
  This simple example shows that the name QK product is rather apt.
   \begin{Rem}
 The metric (\ref{sinecone}) is closely related to the so-called sine-cone construction: if $\mathcal{M}$ is a positive scalar curvature Einstein metric then its sine-cone is the product space $[0,\pi]\times \mathcal{M}$ endowed with the metric
 \begin{equation}
 	d\theta^2+\sin^2(\theta)g_{M}.\label{scmetric}
 \end{equation}
Although these metrics are singular at $\theta=0,\pi$ (except for round spheres) the singularities are rather mild as they are only conical singularities. In \cite{Foscolo2015a} Foscolo-Haskins found exotic Einstein metrics on $S^6$ and $S^3\times S^3$ by desingularising suitable sine-cones. The sine-cone metric (\ref{scmetric}) can be viewed as a special case of (\ref{sinecone}) when $\mathcal{S}_1$ is taken to be a point.
   \end{Rem}
\section{A HK/QK type correspondence}\label{HKQKsection}

The goal of this section is to prove that the QK reduction using a permuting Killing vector field $X$ as in subsection \ref{permutingsubsection} can in fact be inverted. More precisely, we can reconstruct the hyperK\"ahler manifold $(M^{4n},\sigma_1,\sigma_2,\sigma_3)$ together with the vector field $X$ from suitable data on $N^{4n}_{red}:=N^{4n+4}/\!\!/\!\!/\!\!/ S^1_{\tilde{X}}$. This is closely related to the construction of Haydys in \cite{Haydys2008} whereby he shows that the HK quotient of the Swann bundle by a (lifted) $S^1$ action can be inverted to recover the original (positive scalar curvature) QK manifold. By contrast our construction gives a negative scalar curvature version of this correspondence, which was precisely given in \cite{Alekseevsky2015, Alekseevsky2013} by different methods.
\subsection{The permuting quotient}\label{thepermutingsubsection}
Before describing the inverse construction we first need to understand better the induced geometric structure on the quotient $N^{4n}_{red}$.

Given a permuting Killing vector field $X$ on $M^{4n}$ we consider the more general lift
\[\tilde{X}=X-2\cdot (y_3 \partial_{y_2}-y_2 \partial_{y_3})-\frac{a}{2}\partial_{y_1}\]
to $N^{4n+4}$. Then from the results of the previous section we know that the QK moment map is given by
\begin{equation}
	f_{\tilde{X}}=e^{2t}((\frac{a}{2}-\frac{1}{2}e^{-2t}-X \ip \kappa_1) \om_1-2y_3 \om_2 +2y_2 \om_3)
\end{equation}
and hence $f_{\tilde{X}}^{-1}(0)=\{y_2=0, y_3=0, t=-\frac{1}{2}\log(a-2(X\ip\kappa_1))\}$. It is worth pointing out that the presence of the constant $a$ is important as we shall see in subsection \ref{examplepermutingsection}.

For the rest of this section we shall work on $f_{\tilde{X}}^{-1}(0)$ and pullback the data on $N^{4n+4}$
to the submanifold $f_{\tilde{X}}^{-1}(0)$. It is worth highlighting that $f_{\tilde{X}}^{-1}(0)$ is diffeomorphic to $M^{4n+1}_{\alpha}$.
By abuse of notation we denote the pullbacked differential forms and metric by the same expression.
The $S^1$ action generated by the Killing vector field $\tilde{X}$ on $f_{\tilde{X}}^{-1}(0)$ can be identified with
\[\tilde{X}\big|_{f_{\tilde{X}}^{-1}(0)}=X-\frac{a}{2}\partial_{y_1}.\]
We shall also denote the latter by $\tilde{X}$ to ease notation. 

\begin{Prop}
	The differential forms $\om_1$ and $\Om$ descend to the quotient $N^{4n}_{red}$.
\end{Prop}
\begin{proof}
 Using the fact that $t=-\frac{1}{2}\log(a-2(X\ip\kappa_1))$ on $f_{\tilde{X}}^{-1}(0)$, a simple computation shows that
 \begin{align*}
 	\tilde{X}\ip \om_1 &=e^{2t}(X\ip \sigma_1)+e^{2t} (a- 2X\ip\kappa_1) dt \\
 	&= -e^{2t}d(X\ip\kappa_1)+e^{2t} d(X\ip\kappa_1)\\
 	&=0.
 \end{align*}
 Together with the first equation of (\ref{qkequ}) one deduces that $\mathcal{L}_{\tilde{X}}\om_1=0$ and hence $\om_1$ passes to $N^{4n}_{red}$. An analogous argument applies to $\Om$.
\end{proof}
The induced QK metric $g_{red}$ on $N^{4n}_{red}$ is determined by
\begin{equation}
	g_\Om= g_\Om(\tilde{X},\tilde{X})\cdot \xi_X^2+g_{red},\label{metriconlevelset}
\end{equation}
where $\xi_X$ is the Riemannian connection $1$-form on $\pi: f_{\tilde{X}}^{-1}(0)\to N^{4n}_{red}$ defined by
\begin{equation}
	\xi_X(\cdot):=\frac{g_{\Om}(\tilde{X},\cdot)}{g_{\Om}(\tilde{X},\tilde{X})}.\label{connectionXix}
\end{equation}
Strictly speaking one should write $\pi^*g_{red}$ in (\ref{metriconlevelset}) but since $\pi$ is a Riemannian submersion we can identify $g_{red}$ with $\pi^*g_{red}$.

The key observation now is that there is another $S^1$ action on $f_{\tilde{X}}^{-1}(0)$ generated by the Killing vector field $\partial_{y_1}$ which also preserves $\xi_X$. Moreover it commutes with $\tilde{X}$ and as such it defines an $S^1$ action on $N^{4n}_{red}$. This will be the crucial ingredient to inverting this quotient. We denote by $Z$ the vector field on $N^{4n}_{red}$ generated by this action.
\begin{Prop}\label{Zhorizontal}
	The Killing vector field $\partial_{y_1}$ can be expressed as
	\begin{equation}
		\partial_{y_1}=Z_h -2(a-2p+g_{M}(X,X))^{-1}\cdot \tilde{X},
	\end{equation}
where $Z_h$ denotes the horizontal lift of $Z$ and $p:=X\ip \kappa_1$ (the negative of the moment map for $X$ associated to $\sigma_1$ on $M^{4n}$).
\end{Prop}
\begin{proof}
	Since by definition $\pi_*(\partial_{y_1})=\pi_*(Z_h)=Z$ we only need to check that $\xi_X(\partial_{y_1})=-2(a-2(X\ip \kappa_1)+g_{M}(X,X))^{-1}$. Observe that on $f_{\tilde{X}}^{-1}(0)$ we have
	\begin{equation}	g_{\Om}=\frac{dp^2}{(a-2p)^2}+4(a-2p)^{-2}((dy_1+\kappa_1)^2+\kappa_2^2+\kappa_3^2)+(a-2p)^{-1}g_{M}\label{permutingquotientmetric}
	\end{equation}
	and from this we compute the connection form explicitly as
	\begin{equation}
		\xi_X=\frac{X^\flat-2 dy_1-2\kappa_1}{a-2p+g_{M}(X,X)},\label{conn}
	\end{equation}
	where $X^\flat:=g_M(X,\cdot)$. It is now easy to see that 
	\begin{equation*}
		\xi_X(\partial_{y_1}) =-2(a-2p+g_{M}(X,X))^{-1}
	\end{equation*}
and this concludes the proof.
\end{proof}
Using (\ref{permutingquotientmetric}) and (\ref{conn}) we can express $g_{red}$ explicitly in terms of the data on $M^{4n}$ as 
\begin{align}
g_{red} =& \ \frac{dp^2}{(a-2p)^2}+\frac{4g_{M}(X,X)}{(a-2p)^2(a-2p+g_{M}(X,X))}(dy_1+\kappa_1)^2\label{gredmetric}\\
&+ \frac{4}{(a-2p)^2}(\kappa_2^2+\kappa_3^2)+\frac{1}{(a-2p)}g_M \nonumber\\
&+\frac{1}{(a-2p)(a-2p+g_{M}(X,X))}(2 X^\flat \odot (dy_1+\kappa_1)-(X^\flat)^2).\nonumber
\end{align}
In contrast to $\om_1$, $\mathcal{L}_{\tilde{X}}\om_i\neq 0$ for $i=2,3$ and hence these do not descend to $N^{4n}_{red}.$ However by choosing new local coordinates on $f_{\tilde{X}}^{-1}(0)$ we can write $\tilde{X}=\partial_x$, where $x$ denotes the fibre coordinate and we define
\begin{equation}\label{om2om3def}
	\bar{\om}_2:=h\cdot \om_2-f\cdot \om_3 \text{ \ \ and \ \ } \bar{\om}_3:=f\cdot \om_2+h\cdot \om_3
\end{equation}
where $f:=\cos(2x)$ and $h:=\sin(2x)$. A simple calculation now shows that:
\begin{Prop}
	The $2$-forms $\bar{\om}_2$ and $\bar{\om}_3$ descend to $N^{4n}_{red}$.
\end{Prop}
In summary, we have shown that $N^{4n}_{red}$ inherits the data $(g_{red},\bar{\Om}:=\Om,\bar{\om}_1:=\om_1,\bar{\om}_2,\bar{\om}_3)$, the curvature $2$-form $d\xi_X$ and the Killing vector field $Z$. It is clear by construction that the triple $\bar{\om}_1, \bar{\om}_2,\bar{\om}_3$ determines the QK structure on $N^{4n}_{red}$. 
Moreover we have:
\begin{Prop}\label{qkequofQuotientprop}
	The $\mathfrak{sp}(1)$-component of the Levi-Civita connection of $(N^{4n}_{red},g_{red})$ is determined by
	\begin{equation}\label{qkequofQuotient}
		\begin{cases}
		\begin{matrix}
			d\bar{\om}_1 &=& 0 &+& -\al_2\w \bar{\om}_2 &+& \al_3\w \bar{\om}_3,\\
			d\bar{\om}_2 &=& \al_2\w \bar{\om}_1 &+& 0  &+& \beta\w \bar{\om}_3,\\
			d\bar{\om}_3 &=& -\al_3\w \bar{\om}_1 &+& -\beta\w \bar{\om}_2 &+& 0,
		\end{matrix}
		\end{cases}
			\end{equation}
where $\al_2:=-4 e^{2t} (f\kappa_2+h\kappa_3)$, $\al_3:=4 e^{2t}(-h\kappa_2+f\kappa_3 )$ and $\beta:=2dx+4e^{2t}\alpha$.
\end{Prop}
\begin{proof}
	First note that from (\ref{om2om3def}) we can write
	\begin{equation}
		{\om}_2=h\cdot \bar{\om}_2+f\cdot \bar{\om}_3 \text{ \ \ and \ \ } {\om}_3=-f\cdot \bar{\om}_2+h\cdot \bar{\om}_3.
	\end{equation}
So from the first equation of (\ref{qkequ}) (again pullbacked to $f_{\tilde{X}}^{-1}(0)$) we have
\begin{align*}
	d\om_1 &= (4e^{2t}\kappa_3)\w (h\cdot \bar{\om}_2+f\cdot \bar{\om}_3 )-(4e^{2t}\kappa_2)\w(-f\cdot \bar{\om}_2+h\cdot \bar{\om}_3)\\
	&= 4e^{2t}(f\kappa_2+h\kappa_3)\w \bar{\om}_2+4e^{2t}(-h\kappa_2+f\kappa_3)\w \bar{\om}_3
\end{align*}
and this proves the first line of (\ref{qkequofQuotient}). The analogous argument applied to $\bar{\om}_2$ and $\bar{\om}_3$ concludes the proof.
\end{proof}
It is worth noting that $\beta(Z)=4(a-2p)^{-1}.$ 
Since $N^{4n}_{red}$ inherits a Killing vector field $Z$, the natural next step is to understand the properties of this action. A simple calculation shows:
\begin{Prop}
The vector field $Z$ preserves the data $(g_{red},\bar{\om}_1,\bar{\om}_2,\bar{\om}_3,d\xi_X)$ on $N^{4n}_{red}$.
\end{Prop}
Using (\ref{gredmetric}) and from the definition of $\beta$, we derive the following explicit expression for the connection form $\xi_X$ in terms of the data on $N^{4n}_{red}$.
\begin{Prop}\label{propconnectiononlevelset}
	On $f^{-1}_{\tilde{X}}(0)$ we have
	\begin{equation}
		\xi_X=dx+\frac{1}{2}(a-2p)Z^\flat-\frac{1}{2}\beta,\label{connectiononlevelset}
	\end{equation}
where $Z^{\flat}:=g_{red}(Z,\cdot)$. In particular, the curvature form $d\xi_X$ is completely determined by $Z$ and the QK structure on $N^{4n}_{red}.$
\end{Prop}
Now observe that $X^{\flat}:=g_{M}(X,\cdot)=g_{M}(I_iX,I_i\cdot)=I_i(X\ip\sigma_i)$ on $M^{4n}$ . Thus, we have
\begin{equation}
	X^{\flat}=-I_1(dp)=-2I_2(\kappa_3)=+2I_3(\kappa_2)
\end{equation}
and in particular, $dX^{\flat}\in \Om^{1,1}_{I_1}(M^{4n})$.
The analogous calculation on $N^{4n}_{red}$ shows that
\begin{equation}
	4Z^{\flat}=-\bar{I}_1(d(\beta(Z)))=\beta(Z)\cdot\bar{I}_2(\al_2)=-\beta(Z)\cdot\bar{I}_3(\al_3),\label{zsharp}
\end{equation}
where $\bar{I}_1,\bar{I}_2,\bar{I}_3$ denote the almost complex structures associated to $\bar{\om}_1,\bar{\om}_2,\bar{\om}_3$. A similar computation as in the proof of Theorem \ref{hypercomplextheorem} then shows that
\begin{Prop}
	$(N^{4n}_{red},\bar{I}_1)$ is a complex manifold.
\end{Prop}
Thus, we deduce that $dZ^{\flat}, d\xi_X \in \Om^{1,1}_{\bar{I}_1}(N^{4n}_{red})$. The results of \cite{GAMBIOLI2015146} assert that up to a constant factor involving the scalar curvature $d(\|(dZ^\flat)^{\mathfrak{sp}(1)}\|)$ is equal to $\bar{I}_1(Z^\flat)$, and by verifying in an example (for instance see subsection \ref{examplepermutingsection} below) we find that
\begin{equation}
	(dZ^\flat)^{\mathfrak{sp}(1)}={4}(a-2p)^{-1}\bar{\om}_1.
\end{equation}
Moreover the QK moment map associated to the action generated by $Z$ on $N^{4n}_{red}$ is given by
\begin{equation}
Z\ip \bar{\Om}=-d({(a-2p)^{-1}}\bar{\om}_1).\label{QKmomentmapforZ}
\end{equation}
\begin{Rem}
	Since $a-2p>0$, from equation (\ref{QKmomentmapforZ}) the QK quotient of $N^{4n}_{red}$ with respect to $Z$ is the empty set.
\end{Rem}
So far we have shown how to define $g_{red}$ starting from $g_{M}$ and suitable data on $M^4$ (see (\ref{gredmetric})). We now give a somewhat converse result.
\begin{Prop}\label{metriconlevelsetprop}
	The HK metric on $M^{4n}$ can be expressed as
\begin{align}
	g_{M}=&\  4\beta(Z)^{-1}g_{red}-16\beta(Z)^{-3}(Z^\flat)^2+8\beta(Z)^{-2}(\xi_X \odot Z^\flat)\label{hkmetric}\\
	& -\beta(Z)^{-1}(\al_2^2+\al_3^2)-\frac{16\|Z\|^2}{4\beta(Z)\|Z\|^2-\beta(Z)^3}(\xi_X)^2-\frac{(Z\ip d\beta)^2}{\beta(Z)^3},\nonumber
\end{align}
where $\|Z\|^2:=g_{red}(Z,Z).$ 
\end{Prop}
\begin{proof}
	From Proposition \ref{Zhorizontal} we have
	\[Z_h=(1-a(a-2p+g_M(X,X))^{-1})\partial_{y_1}+2(a-2p+g_M(X,X))^{-1}X\]
	and from Proposition \ref{propconnectiononlevelset} we have $$\xi_X=-2(a-2p)^{-1}\al+\frac{1}{2}(a-2p)Z^\flat.$$ Since $\xi_X(Z_h)=0$, one computes  \begin{equation}
		g_M(X,X)=(a-2p)^3\|Z\|^2(4-(a-2p)^2\|Z\|^2)^{-1}.\label{positiveHK}
	\end{equation} 
	Using the latter, (\ref{permutingquotientmetric}) and 
	\begin{equation}
		\al_2^2+\al_3^2=16(a-2p)^{-2}(\kappa_2^2+\kappa_3^2)
	\end{equation}
	one can rewrite (\ref{gredmetric}) as (\ref{hkmetric}).
\end{proof}
Propositions \ref{qkequofQuotientprop}, \ref{propconnectiononlevelset} and \ref{metriconlevelsetprop} are all suggestive that one should be able recover the HK structure of $M^{4n}$ from the QK structure of $N^{4n}_{red}$ and $Z$. Before proving that this is indeed the case, we first describe a couple of concrete examples illustrating the above reduction.
\subsection{Examples}\label{examplepermutingsection}
\subsubsection{Example 1.}
As in subsection \ref{explicithomotheticexample} we again take $M=\R^4$ with its standard HK structure but we now consider the permuting vector field
\[X=-2x_4 \partial_{x_3}+2x_3 \partial_{x_4}.\]
This corresponds to rotation on the $\R^2$ factor spanned by $x_3,x_4$.
We can choose the connection $1$-form 
\begin{gather*}
	\al=dy_1-\frac{1}{2}(x_2dx_1-x_1dx_2+x_4dx_3-x_3dx_4)
\end{gather*}
as before, since it is also preserved by $X$, and by definition we have 
\begin{equation*}
	\xi=dy_2+x_3dx_1-x_4dx_2 \text{\ \ \ \ and\ \ \ \ }
	\eta=dy_3+x_4dx_1+x_3dx_2.
\end{equation*}
Applying the above construction, from (\ref{connectionXix}) one finds that
\[\xi_X=-(2x_3^2+2x_4^2+a)^{-1}(2dy_1-x_2dx_1+x_1dx_2+x_4dx_3-x_3dx_4).\]
Writing $p:=X \ip \kappa_1 = x_3^2+x_4^2$ and $q:=4y_1-a \arctan(\frac{x_3}{x_4})$, we find after a long computation that
\begin{align}
g_{red} =& \ \frac{(2p+a)}{(2p-a)^2}(dx_1^2+dx_2^2)+\frac{(2p+a)}{4p(2p-a)^2}(dp^2)\label{redQKmetricR4}\\
& +\frac{p}{(2p-a)^2(2p+a)}(dq-2(x_2dx_1-x_1dx_2))^2.\nonumber
\end{align}
From this expression we see that $N^{4}_{red}$ can be viewed an $S^1_q$ bundle over $\R^3$ with coordinates $(x_1,x_2,p)$. The vector field $\partial_q$ essentially corresponds to the Killing vector field $Z$.
One can now verify directly that $g_{red}$ is indeed a self-dual Einstein metric (with scalar curvature $-48$). 
Note that although the QK reduction is only valid for $a-2p>0$ (since we require $t=-\frac{1}{2}\log(a-2p)$) the metric $g_{red}$ is well-defined even if $a=0$.
When $a=0$, a change of coordinates shows that 
\[g_{red}=4e^{s}(dx_1^2+dx_2^2)+ds^2+e^{2s}(dq-2(x_2dx_1-x_1dx_2))^2\]
which one recognises is the complex hyperbolic metric. In particular, the latter is a K\"ahler metric but when $a \neq 0$ the metrics $g_{red}$ are not K\"ahler.
\begin{Rem}
The author would like to thank the referees for pointing out to us that the metric (\ref{redQKmetricR4}) was also recently shown to arise as the one-loop deformed complex hyperbolic plane in \cite[Sect. 4.3]{Cortes2022}, which is also an application of the construction in \cite{Alekseevsky2015,Alekseevsky2013}.
\end{Rem}
\subsubsection{Example 2.}
If instead we take the permuting Killing vector field to be
\[\check{X}=-x_2\partial_{y_1}+x_1\partial_{x_2}-x_4\partial_{x_3}+x_3\partial_{x_4},\]
which corresponds to rotation on both $\R^2_{x_1,x_2}$ and $\R^2_{x_3,x_4}$, then one finds that 
\[\xi_{\check{X}}=-{2}{a^{-1}}dy_1.\]
After another long calculation we have be able to show that $g_{red}$ corresponds to the QK hyperbolic metric on $\R^4$. More concretely, the metric takes the form
\[g_{red}=\frac{a}{(a-r^2)^2}\cdot g_{\R^4},\]
where $\R^4$ is viewed as an $S^1$ bundle over $\R^3$ as in the Gibbons-Hawking ansatz with the $S^1$ action generated by $2a^{-1}Z$ i.e. $\theta(2a^{-1}Z)=1$, \[\mathbf{u}=\Big(\frac{1}{2}(x_1^2+x_2^2-x_3^2-x_4^2),\ x_1x_3+x_2x_4,\ x_1x_4-x_2x_3\Big)\] 
and the harmonic function $V=\frac{1}{2|\mathbf{u}|}=\frac{1}{r^2}$. 
The metric becomes singular for $a=0$ but nonetheless it converges (in the pointed Cheeger-Gromov sense) to the flat metric on $\R^4$ as $a \to 0$.

More generally one can apply the above construction to the case when $M=\R^{4n}$ by extending the permuting Killing vector fields $X$ and $\check{X}$ to $\R^{4n}$ in the obvious way though identifying the quotient metrics remains a challenging task.

On $\R^{4n}$ defining a permuting Killing vector field amounts to choosing a $U(1) $ subgroup of $Sp(n)U(1)\cong U(2n) \cap Sp(n)Sp(1)$, that is not properly contained in the $Sp(n)$ factor. Observe that this was the case in the above examples. So in general, there is an $Sp(n)$ family of such $U(1)$ actions.

In particular, by choosing $1$-parameter families of such $U(1)$ actions one should be able to construct $2$-parameter families ($1$-parameter coming from the constant $a$) of QK metrics. Concretely in the above case this means that one can consider the permuting Killing vector field
\[s\cdot X+(1-s)\cdot \check{X}\]
for any $s\in \R$. This again satisfies the hypothesis of subsection \ref{permutingsubsection} and hence we can apply the QK reduction as described above. 
We expect this will give a family of QK metrics  
connecting the complex hyperbolic metric to the quaternion hyperbolic one, though finding an explicit expression for $g_{red}(s)$ seems rather hard. It is worth emphasising that the definitions of $\kappa_i$ will also depend on $s$.

\subsection{The inverse construction}
Suppose now that $(\bar{N}^{4n},\bar{\Om},\bar{g})$ is an \textit{arbitrary} QK manifold. Then by choosing locally $\bar{\om}_i$ we have that (\ref{qkequofQuotient}) holds for some $1$-forms $\beta, \al_2,\al_3$. Moreover, we have that
	\begin{equation}
	\begin{cases}
		\begin{matrix}
			d\beta &=& -4s \bar{\om}_1 &+& \al_2\w \al_3,\\
			d{\al}_2 &=& +4s \bar{\om}_3 &+&  \al_3\w \beta,\\
			d{\al}_3 &=& +4s \bar{\om}_2 &+& \beta\w \al_2,
		\end{matrix}\label{connec}
	\end{cases}
\end{equation}
where by rescaling $(\bar{\Om},\bar{g})$ we can set $s=+1$ if the scalar curvature is positive and $s=-1$ if it is negative cf. \cite{GAMBIOLI2015146}. The above is essentially a consequence of the fact that the induced Levi-Civita connection on $\langle \bar{\om}_1,\bar{\om}_2,\bar{\om}_3\rangle$ has constant curvature. Note that differentiating (\ref{connec}) one indeed recovers (\ref{qkequofQuotient}). We will henceforth assume that $s=-1$ i.e. $scal(\bar{g})<0$.

Let $Z$ be a Killing vector field on $\bar{N}^{4n}$ such that $\mathcal{L}_Z\bar{\Om}=0$ and such that the projection $(dZ^\flat)^{\mathfrak{sp}(1)}$ does not vanish everywhere. We can then define
\begin{equation}
	\bar{\om}_1:=(2n)^{1/2}\cdot\frac{(dZ^\flat)^{\mathfrak{sp}(1)}}{\|(dZ^\flat)^{\mathfrak{sp}(1)}\|}
\end{equation}
on a suitable open set. It follows that $\mathcal{L}_Z \bar{\om}_1=0$.
We shall further assume that we can choose $\bar{\om}_2,\bar{\om}_3$ such that $\mathcal{L}_Z \bar{\om}_2=\mathcal{L}_Z\bar{\om}_3=0$ and that $\beta(Z)>2\|Z\|$.

Applying $\mathcal{L}_Z$ to (\ref{qkequofQuotient}) and using the quaternions relations we get \begin{equation}
	\mathcal{L}_Z(\beta)=\mathcal{L}_Z(\al_2)=\mathcal{L}_Z(\al_3)=0.
\end{equation}
From this one deduces that (\ref{zsharp}) holds and that $\bar{I}_1$ is integrable (this was also shown in \cite[Sec. 3]{GAMBIOLI2015146}). If $[d(2\beta(Z)^{-1}Z^\flat-\frac{1}{2}\beta)]\in H^2(\bar{N},\mathbb{Z})$, we define the connection $1$-form
\[\xi_X:=dx+2(\beta(Z)^{-1})Z^\flat-\frac{1}{2}\beta\]
on the total space of this $S^1$ bundle, where $x$ denotes the coordinate on the fibre. This is consistent with (\ref{connectiononlevelset}) . 

We denote by $\tilde{Z}$ the lift of $Z$ to the total space satisfying $dx(\tilde{Z})=0$ and we define $M^{4n}$ as the quotient obtained by the action generated by $\tilde{Z}$. Observe that $\mathcal{L}_{\tilde{Z}}\xi_X=0$ by construction. The reader should think of $\tilde{Z}$ as corresponding to the vector field $\partial_{y_1}$ of subsection \ref{thepermutingsubsection}.

\begin{Th}\label{HKQKcorrespondenceTheorem}
The closed $2$-forms
	\begin{gather}
		\begin{cases}
		\sigma_1 := -d(\beta(Z)^{-1}(2\cdot dx-\beta)),\\
		\sigma_2 :=-d(\beta(Z)^{-1}(h\cdot \al_3+f\cdot \al_2)),\\
		\sigma_3 :=-d(\beta(Z)^{-1}(h\cdot\al_2-f\cdot \al_3)),
		\end{cases}
	\end{gather}
where $f:=\cos(2x)$ and $h:=\sin(2x)$, descend to $M^{4n}$ to define a HK triple. Moreover, the vector field $\partial_x$ induces an $S^1$ action on $M^{4n}$ preserving $\sigma_1$ and permuting $\sigma_2$ and $\sigma_3$.
\end{Th}
\begin{proof}
	It is easy to see that $\mathcal{L}_{\partial_x}\sigma_1=0$, $\mathcal{L}_{\partial_x}\sigma_2=-2\sigma_3$ and $\mathcal{L}_{\partial_x}\sigma_3=+2\sigma_2$. The fact that $\sigma_i$ satisfy the quaternionic relation is an algebraic condition and can be easily deduced by reversing the steps in subsection \ref{thepermutingsubsection} with $\beta(Z)=4(a-2p)^{-1}$.
	So the only remaining thing to check is that $\tilde{Z}\ip \sigma_i=0$. First we have that
	\[\tilde{Z}\ip\sigma_1=-\mathcal{L}_{\tilde{Z}}(\beta(Z)^{-1}(2\cdot dx-\beta))-d(\beta(Z)^{-1}\beta(Z))=0.\]
	Secondly note that from (\ref{zsharp}) we know that $Z\ip\al_2=Z\ip\al_3=0$ and hence
	\[\tilde{Z}\ip\sigma_2=-\mathcal{L}_{\tilde{Z}}(\beta(Z)^{-1}(h\cdot \al_3+f\cdot \al_2))=0\]
	and 
	\[\tilde{Z}\ip\sigma_3=-\mathcal{L}_{\tilde{Z}}(\beta(Z)^{-1}(h\cdot \al_2-f\cdot \al_3))=0.\]	
Thus, $\sigma_i$ are basic $2$-forms (with respect to the action generated by $\tilde{Z}$) and hence descend to $M^{4n}$.
\end{proof}
\begin{Rem}
The above definition of $\sigma_i$ was chosen to match the previous condition that $\sigma_i=d\kappa_i$. So this indeed corresponds to the inverting the construction described in subsection \ref{thepermutingsubsection}.
\end{Rem}
It also follows now that the induced HK metric on $M^{4n}$ is given by (\ref{hkmetric}). One can verify directly that indeed $g_{M}(\tilde{Z},\cdot)$ vanishes. The hypothesis that $\beta(Z)>2\|Z\|$ ensures that $g_M$ is positive definite i.e. $g_M(\bar{I}_i(Z),\bar{I}_i(Z)), g_{M}(X,X) >0$ (see (\ref{hkmetric}) and (\ref{positiveHK})).
\begin{Rem}
It was brought to our attention by Vicente Cort\'es that the correspondence established by Theorem \ref{HKQKcorrespondenceTheorem} was also demonstrated in \cite{Alekseevsky2015, Alekseevsky2013} albeit from a rather different approach. The latter works were motivated by the c-map construction, originating from the work of Ferrara-Sabharwal in \cite{Ferrara1990}, and involves making the correspondence via a certain pseudo-HK manifold
whereas our approach here stays in the realm of Riemannian geometry and goes via the QK manifold $N^{4n+4}$. To the best of our knowledge the presence of the Einstein manifold $N^{4n+4}$ does not seem to have been predicted in the physics literature so it would be interesting to understand the role of $N$ from a physics perspective.
The authors of \cite{Alekseevsky2015, Alekseevsky2013} considered more generally the case when $g_M$ and $g_{red}$ also have mixed signature (though our construction can naturally be extended to these cases as well).
\end{Rem}
\noindent\textbf{Analogy with the Haydys HK/QK correspondence.}\\
\indent In \cite{Swann1991} Swann shows that given a QK manifold $N^{4n}_{+}$ with positive scalar curvature, one can construct an associated bundle $\mathcal{U}(N_+)$ using the action of $Sp(n)Sp(1)$ on $\mathbb{H}^\times/\mathbb{Z}_2$ and moreover it admits a natural HK structure. Now if $N_+$ admits a Killing vector field $Y$ preserving the QK structure then he shows that one can lift this action to a triholomorphic Killing action on $\mathcal{U}(N_+)$. On the other hand there is also an isometric $S^3 \subset \mathbb{H}^\times$ action on the fibre of $\mathcal{U}(N_+)$ commuting with the lifted $S^1_{\tilde{Y}}$ action and this descends to the HK quotient $M^{4n}:=\mathcal{U}(N_+)/\!\!/\!\!/\!\!/ S^1_{\tilde{Y}}$ to give a permuting Killing vector field $X$. In \cite{Haydys2008} Haydys then shows that one can in fact invert this construction to recover $N_+$ from $M$, see also \cite{Hitchin2013} for a twistorial description. The construction described in this section shows that $(M,X)$ is also in a one-to-one correspondence with $(\bar{N},Z)$. This is summed up in the following diagram:
\begin{center}
	\begin{tikzcd}[column sep=tiny]
		&\ \ \  S^1_{\tilde{Y}} \curvearrowright \mathcal{U}(N_+) \curvearrowleft S^3  \ar[dr] \ar[dl]
		& &   S^1_{\tilde{X}} \curvearrowright N  \curvearrowleft \mathbb{T}^3_{y_i} \ar[dr] \ar[dl] & &  \\
		(N_+^{4n},Y) 
		&
		& (M^{4n},X) & \ & (\bar{N}^{4n},Z) & 
	\end{tikzcd}
\end{center}
\begin{Rem}
It is likely that 
$\mathcal{U}(N_+)$ and $N$ are also related by our correspondence say by fixing a permuting $S^1\subset S^3$ on $\mathcal{U}(N_+)$ and a corresponding $S^1 \subset {Sp(1,1)}$ on $N$ (recall that $\frac{Sp(1,1)}{Sp(1)Sp(1)}$ is the fibre of $N\to M$). 
Furthermore we expect that this correspondence can more generally be extended to the hypercomplex manifolds of Theorems \ref{hypercomplextheorem} and \ref{extendedhypercomplextheorem}, and certain complex manifolds using Joyce's hypercomplex reduction \cite{Joyce1991}. This will be investigated in future work.
\end{Rem}
\begin{Rem}
	It was shown in \cite{BoyerGalickiMann1993} that associated to $N^{4n}_+$ are three other positive scalar curvature Einstein manifolds, namely $\mathcal{U}(N_+)$, $Z(N_+)$ and $\mathcal{S}(N_+)$. By analogy associated to $M^{4n}$, there are three negative scalar curvature Einstein manifolds, namely $N$, $P$ and $L$ of Theorem \ref{maintheorem}. Moreover in each case $Z(N_+)$ and $P$ are in fact K\"ahler-Einstein. It would be interesting to find directly a correspondence between these Einstein manifolds as well. 
\end{Rem}

\section{Twistor spaces}\label{twistorsection}
In this section we construct the twistor spaces of the QK manifolds $N^{4n+4}$ of Theorem \ref{maintheorem} and that of the hypercomplex manifolds $M^{4n+4}_{\al,\xi,\eta,\nu}$ of Theorem \ref{hypercomplextheorem}. 

The twistor space is topologically defined as the unit sphere bundle in the rank $3$ vector bundle $\mathcal{G}$. In general this bundle is not trivial; for instance the twistor space of $\mathbb{H}\mathbb{P}^n$ is $\C \mathbb{P}^{2n+1}$. However in our case we do have globally well-defined hypercomplex structures so that topologically $Z(N)$ and $Z(M_{\al,\xi,\eta,\nu})$ are diffeomorphic to $N\times S^2$ and $M_{\al,\xi,\eta,\nu}\times S^2$ respectively. 
Both the twistor space of a QK manifold and hypercomplex manifold admit a distinguished complex structure (not a product one). The idea of twistor theory is essentially to encode the data on $N$ and $M_{\al,\xi,\eta,\nu}$ in terms of the complex geometry of their twistor spaces \cite{AtiyahHitchinSinger78, Salamon1986}. We shall first consider the hypercomplex twistor space then the QK one following closely \cite{HKLR} and \cite{Salamon1982}. 

\subsection{Twistor space of hypercomplex examples} 
We now define the complex structure on the twistor space $Z(M_{\al,\xi,\eta,\nu})$.
Consider the unit $2$-sphere $S^2=\{(a,b,c)\ |\ a^2+b^2+c^2=1\}$, then  
at the point $p=(m,(a,b,c))\in M_{\al,\xi,\eta,\nu} \times S^2$ we define an almost complex structure $\tilde{I}$ on $T_{p}Z(M_{\al,\xi,\eta,\nu})\cong T_mM_{\al,\xi,\eta,\nu} \oplus T_{(a,b,c)}S^2$ by
\begin{align*}
	\tilde{I}=&\ (aI_1+bI_2+cI_3,I_0)\\
	=&\ \Big(\frac{ 1-z\bar{z}}{1+z\bar{z}}I_1+\frac{ z+\bar{z}}{1+z\bar{z}}I_2+\frac{ i(z-\bar{z})}{1+z\bar{z}}I_3,I_0\Big)
\end{align*}
where $z$ denotes a local complex coordinate on $S^2 \cong \C \mathbb{P}^1$ and $I_0$ is the associated complex structure. To prove that $\tilde{I}$ is integrable we first need to identify the $(1,0)$ forms on $Z(M_{\al,\xi,\eta,\nu})$ aside from $dz$. 

As in the hyperK\"ahler case \cite{HKLR}, it is easy to check that if $\theta$ is a $(1,0)$ form with respect to $I_1$ then $\theta+zI_3(\theta)$ is a $(1,0)$ form with respect to $\tilde{I}$ on $Z(M_{\al,\xi,\eta,\nu})$. Thus, a basis of $(1,0)$ forms for $I_1$ together with $dz$ give a basis of $(1,0)$ form for $\tilde{I}$.
Integrability of $\tilde{I}$ will follow if we can show that $d(\Gamma(\Lm^{1,0}))\subset \Gamma(\Lm^{2,0}\oplus \Lm^{1,1})$.

For any torsion-free connection $\nabla$ on $M_{\al,\xi,\eta,\nu}$ we have that 
\[d(\theta+zI_3(\theta)) = dx_i \w \nabla_{\partial_{x_i}}(\theta+zI_3(\theta) )+dz \w I_3(\theta) ,\]
where $\{x_i\}$ denote local coordinates on $M_{\al,\xi,\eta,\nu}$, since $d(\cdot)=\text{Alt}(\nabla(\cdot))$ on differential forms cf. \cite[Corollary 8.6]{KobayashiNomizu1}. Now let $\nabla$ be the Obata connection i.e. the unique torsion-free connection which preserves $I_1,I_2,I_3$ \cite{Obata1956} so that $\nabla \circ \tilde{I}= \tilde{I} \circ \nabla$. Since
\[\tilde{I}(\nabla_{\partial_{x_i}}(\theta+zI_3(\theta) ))=i(\nabla_{\partial_{x_i}}(\theta+zI_3(\theta) ))\]
we deduce that $dz$ and $\nabla_{\partial_{x_i}}(\theta+zI_3(\theta) )$ are both of type $(1,0)$ with respect to $\tilde{I}$ and it follows that $d(\theta+zI_3(\theta)) \in \Lm^{2,0}\oplus \Lm^{1,1}$ i.e. $N_{\tilde{I}}\equiv 0$. 
\begin{Rem}
	For hyperK\"ahler manifolds the Obata connection coincides with the Levi-Civita connection whereas for QK manifolds the difference between these two connections is essentially determined by (\ref{qkequ}).
\end{Rem}
The projection map
\begin{equation}
	p:Z(M_{\al,\xi,\eta,\nu})\to \C \mathbb{P}^1\label{fibrationP1}\end{equation}
is clearly holomorphic and each fibre $p^{-1}(z)\simeq M_{\al,\xi,\eta,\nu}$ can be viewed as a complex manifold endowed with the complex structure $aI_1+bI_2+cI_3$. The antipodal map on $S^2$ induces an anti-holomorphic involution on $Z(M_{\al,\xi,\eta,\nu})$ compatible with $p$.

There is also a globally well-defined $(2,0)$-form on $Z(M_{\al,\xi,\eta,\nu})$ given by
\begin{equation}\check{\Om}:=(\check{\om}_2+i\check{\om}_3)+2z \check{\om}_1 - z^2(\check{\om}_2-i\check{\om}_3).\end{equation}
Unlike in the hyperK\"ahler case however $\check{\Om}$ is not holomorphic on the fibres $p^{-1}(z)$ but the $(2n+2,0)$ form $\check{\Om}^{n+1}$ is. In particular, this implies that each fibre of (\ref{fibrationP1}) has trivial first Chern class.

Consider the $(1,0)$ forms $\xi+i\eta$ and $\nu + i\al$ with respect to $I_1$ on $M_{\al,\xi,\eta,\nu}$, then from the above argument we have that 
\begin{align*}
\chi :=&\ \xi+i\eta+zI_3(\xi+i\eta)+iz(\nu + i\al+zI_3(\nu + i\al))\\
=&\ (\xi+i\eta)-2z\al-z^2(\xi-i\eta)
\end{align*}
is a $(1,0)$ form with respect to $\tilde{I}$. Globally $\chi$ can be interpreted as an $\mathcal{O}(2)$-valued $1$-form on $Z(M_{\al,\xi,\eta,\nu})$ since it depends quadratically on $z$. 

We claim that $\chi$ is a holomorphic connection $1$-form on $p^{-1}(z)$. Computing fibrewise we have 
\[d\chi=-(\sigma_2+i\sigma_3)-2z\sigma_1+z^2(\sigma_2-i\sigma_3),\]
which is indeed a holomorphic $(2,0)$ form on the fibres of $Z(M)\to \C \mathbb{P}^1$, in fact it covariantly constant with respect to the Levi-Civita connection of $g_M$. Moreover, the form $d\chi$ on $Z(M)$ completely determines the HK metric $g_M$ cf. \cite[Theorem 3.3]{HKLR}.

To sum up we have shown that Theorem \ref{hypercomplextheorem} can be described in terms of the twistor theory:
\begin{Th}
The $\mathbb{T}^4$-bundle 
\[Z(M_{\al,\xi,\eta,\nu}) \to Z(M)\] 
is holomorphic with vertical $(1,0)$ forms spanned by $z_1(\xi+i\eta)-z_2(\al+i\nu)$ and $z_1(\nu + i\al)+z_2(\eta + i\xi)$ for $[z_1:z_2]\in \C \mathbb{P}^1$.
Furthermore, $\chi$ is a holomorphic section of $\Lm^{1}Z(M_{\al,\xi,\eta,\nu})\otimes \mathcal{O}(2)$ such that $d\chi\in \Om^{2}(Z(M))(2)$ defines a holomorphic symplectic form on the fibres of $Z(M) \to \C \mathbb{P}^1.$
\end{Th}

\subsection{Twistor space of QK examples}
We now construct the complex structure on the QK twistor space $Z(N)$ of $N$. It will be more convenient to work on the $\C^2$ bundle $\mathbf{H}:=N \times \C^2$ whose projectivisation is just  $Z(N)$ \cite{Salamon1982}. Since $\mathfrak{su}(2)\cong \mathfrak{so}(3)$ we can lift the ${SO}(3)$ connection form (\ref{qkequ}) on  $\langle \om_1,\om_2,\om_3\rangle$ to the $SU(2)$ connection 
\begin{equation}
A:=	2 e^{2t}
\begin{pmatrix}
	-i\al & -\xi+i \eta \\
	\xi+i \eta & i\al 
\end{pmatrix}
\end{equation} 
on $\mathbf{H}$. This is simply the connection induced by the Levi-Civita connection of $g_N$ on the associated vector bundle $\mathbf{H}$. Indeed one verifies directly that the curvature $2$-form $F_A:=dA+A \w A$ is given by
\begin{equation}
		-2 
	\begin{pmatrix}
		i\om_1 & -\om_2+i \om_3 \\
		\om_2+i \om_3 & -i\om_1 
	\end{pmatrix}.
\end{equation} 
In \cite{Salamon1982} Salamon shows that $\mathbf{H}$ admits an integrable complex structure with respect to which the vertical $1$-forms 
\begin{align}
	\theta_1 &:=\ dz_1+2e^{2t}(-iz_1\al+z_2(-\xi+i\eta)) ,\\
	\theta_2 &:=\ dz_2+2e^{2t}(+iz_2\al+z_1(+\xi+i\eta))
\end{align}
are of type $(1,0)$, where $z_1,z_2$ denote the coordinates of the fibre $\C^2$. This complex structure is $\C^\times$-invariant and hence passes down to $Z(N)$. Furthermore, the $(1,0)$ form
\begin{align*}
\Theta :=&\ z_2\theta_1-z_1\theta_2\\
=&\ z_2dz_1-z_1dz_2-2e^{2t}(z_1^2(\xi+i\eta)+2iz_1z_2\al  +z_2^2(\xi-i\eta))
\end{align*}
is holomorphic and defines a complex contact structure on $Z(N).$ Denoting by $z=-iz_{1}^{-1}z_2$ a local coordinate on $\C\mathbb{P}^1$ we can express the induced contact structure $\Theta$ on $Z(N)$ as
\[-idz-2e^{2t}\chi.\]
Standard twistor theory now asserts that:
\begin{Th}
$\Theta$ together with the data of the anti-holomorphic involution generated by the antipodal map on the $S^2$ fibres of $Z(N)$ completely determine the QK manifold $(N,\Om_{N},g_N)$. 
\end{Th}
As already seen before when $M=\R^{4n}$ we have that $N=\frac{Sp(n+1,1)}{Sp(n+1)Sp(1)}$. In this case a concrete description of $(Z(N),\Theta)$ was given by LeBrun in \cite[Sec. 2]{LeBrunInfiniteQK} as an open set of $\C \mathbb{P}^{2n+3}$. Using this LeBrun was able to deform the quaternion hyperbolic metric to show the existence of infinitely many QK metrics of negative scalar curvature.
In general however it is hard to identify the complex manifold $Z(N)$ explicitly.

\section{Other examples}\label{furtherexamples}
The purpose of this section is to describe other closely related geometric structures that arise on the topological spaces that figure in Theorems \ref{maintheorem}, \ref{hypercomplextheorem} and \ref{extendedhypercomplextheorem}, and to highlight some links between these various geometries.

\subsection{Ricci-flat examples}
We recall some known solutions to the system (\ref{firstequ})-(\ref{lastequ}) when $\lambda=0$. All of these examples have special holonomy.

\subsubsection{Calabi-Yau metrics}
We first consider metrics on $P^{4n+2}$. From (\ref{riccikahlerpde}) we see that if we set 
\begin{equation}
p(t)=t^{1/(2n+2)}
\end{equation} 
then we get a Ricci-flat metric on $\R^+_t \times M^{4n+1}_\al$. This metric degenerates as $t \to 0$ but is complete as $t\to +\infty$. This is a special case of the so-called Calabi ansatz metrics on complex line bundles \cite{Calabi1957}. Although incomplete, these metrics provide a good approximation for the asymptotic behaviour of the complete Tian-Yau metrics \cite{TianYau1990} and has recently been employed in \cite{hsvz} to study new degenerations of hyperK\"ahler metrics on K3 surfaces.

\subsubsection{$G_2$ and $Spin(7)$ metrics}
Next we consider metrics on $L^7$ and $N^8$. When $n=1$ so that $M$ is a hyperK\"ahler $4$-manifold then Apostolov-Salamon found the solution 
\begin{equation}\label{apostolovsalamonsolution}
g=t^2(t+b)^2dt^2+t^{-2}\al^2+(t+b)^{-2}\xi^2+t(t+b)g_M
\end{equation}
on $L^7$ with holonomy group $G_2 \subset SO(7)$ \cite{Apostolov2003} and in \cite{UdhavFowdar2} we extend their construction and found the solution
\begin{equation}
g=t^2(t+b)^2(t+c)^2dt^2+t^{-2}\al^2+(t+b)^{-2}\xi^2+(t+c)^{-2}\eta^2+t(t+b)(t+c)g_M
\end{equation}
on $N^8$ with holonomy group $Spin(7)\subset SO(8)$, where $b, c$ are positive constants. Like the above Calabi-Yau metrics, these examples are only complete when $t \to +\infty$. We are unaware of any other Ricci-flat solution to (\ref{firstequ})-(\ref{lastequ}). 

\subsection{Balanced Hermitian examples}
We now construct balanced Hermitian structures on $M^{4n+2}_{\xi,\eta}$, $N^{4n+4}$ and $M^{2n+4}_{\nu_1,\nu_2,\nu_3,\nu_4}$. In the latter two cases we show that the Hermitian metric is balanced with respect to the $2$-sphere of complex structures. 

\subsubsection{Balanced structures on $M^{4n+2}_{\xi,\eta}$}
Consider the natural $U(2n+1)$-structure on $M^{4n+2}_{\xi,\eta}$ defined by
\begin{gather}
g=\xi^2+\eta^2+g_M,\label{balancedmetric}\\
\om=\xi \w \eta + \sigma_1.
\end{gather}
Then the $(2n+1,0)$-form defined by 
\begin{equation}
	\Upsilon = (\xi + i\eta) \w (\sigma_2 + i\sigma_3)^n.\label{balancedvolform}
\end{equation}
is closed, and hence holomorphic, since
\[d\Upsilon=-(\sigma_2+i\sigma_3)^{n+1}=0.\]
So $(M^{4n+2}_{\xi,\eta},\Upsilon)$ is a complex manifold with vanishing first Chern class. We also have that
\begin{equation*}
	d(\om^{2n})  = (2n) (-\sigma_2 \w \eta + \sigma_3 \w \xi) \w \sigma_1^{2n-1} =0,
\end{equation*}
where the last equality follows from the fact that $\sigma_2 \w \sigma_1^{2n-1} $ and $\sigma_3 \w \sigma_1^{2n-1} $ are differential forms of type $(2n+1,2n-1)+(2n-1,2n+1)$ on $M^{4n}$. Thus, we have shown 
\begin{Prop}
$(M_{\xi,\eta}^{4n+2}, g, \om ,\Upsilon)$ is a balanced Hermitian manifold with trivial canonical bundle.
\end{Prop}
It is clear from (\ref{balancedmetric}) and (\ref{balancedvolform}) that the $\mathbb{T}^2$ fibres of $M_{\xi,\eta}^{4n+2}$ are elliptic curves and the restricted metric on the fibres is flat. 
\begin{Rem}
The metric $g$ is of course not Calabi-Yau i.e. the holonomy group of the Levi-Civita connection is not a subgroup of $SU(2n+1)$. However, there exists a unique Hermitian connection with totally skew-symmetric torsion form called the Bismut connection \cite{Bismut1989} whose (restricted) holonomy group contained in $SU(2n+1)$ cf. \cite[Proposition 3.6]{MarioLectureonStromingerSystem}. 
\end{Rem}
\begin{Rem}
When $M$ is a hyperK\"ahler $4$-manifold, the $SU(3)$-structure on $M^6_{\xi,\eta}$ determined by $(\om,\Upsilon)$ is both complex and half-flat i.e. $d(\om^2)=0$ and $d\text{Re}(\Upsilon)=0$. If we reverse the complex structure on the $\mathbb{T}^2$-fibres so that the $(2n+1,0)$-form is now given by
\begin{equation}\tilde{\Upsilon}=(\eta + i\xi) \w (\sigma_2 + i\sigma_3)\label{volhalfflat}\end{equation}
and 
\begin{equation}\tilde{\om}=\eta \w \xi + \sigma_1\label{symplechalfflat}\end{equation}
then the $SU(3)$-structure is still half-flat but it is no longer complex; this bears a certain resemblance to the theory of twistor spaces of self-dual $4$-manifolds cf. \cite[Proposition 7.5]{Salamon1989}. In \cite{Apostolov2003,ChiossiSalamonIntrinsicTorsion} the $SU(3)$-structure defined by (\ref{volhalfflat}) and (\ref{symplechalfflat}), when $M=\mathbb{T}^4$, was evolved via the Hitchin flow \cite{Hitchin01stableforms} to construct the $G_2$ holonomy metrics (\ref{apostolovsalamonsolution}).
\end{Rem}
\subsubsection{Balanced structures on $N^{4n+4}$}
Consider the $Sp(n+1)$-structure on $N^{4n+4}$ defined by 
\begin{equation}
\tilde{\om}_1:=e^{-\frac{4}{2n+1}t}\om_1,\ \ \ \ \tilde{\om}_2:=e^{-\frac{4}{2n+1}t}\om_2,\ \ \ \ \tilde{\om}_3:=e^{-\frac{4}{2n+1}t}\om_3,
\end{equation}
where $\om_i$ are as in Theorem \ref{maintheorem}. This simply corresponds to conformally rescaling the quaternion-K\"ahler structure in Theorem \ref{maintheorem}. Since complex structures are conformally invariant it follows that $(N,\tilde{g}:=e^{-\frac{4}{2n+1}t}g_N, \tilde{\om}_i)$ is a Hermitian manifold for $i=1,2,3$.
We leave it to the reader to verify, using (\ref{qkequ}), that
\[d(e^{kt}\om_1^{2n+1})=4(2n+1)(k+4)e^{4(n+1)t+at}\sigma_1^{2n} \w \xi \w \eta\w dt\]
and that one gets analogous expressions for $\om_2$ and $\om_3$. The desired result then follows by taking $k=-4$:
\begin{Prop}
	$(N^{4n+4}, \tilde{g}, a\tilde{\om}_1+b\tilde{\om}_2+c\tilde{\om}_3)$ is a balanced Hermitian manifold for $(a,b,c) \in S^2.$
\end{Prop}
\subsubsection{Balanced structures on $M^{4n+4}_{\nu_1,\nu_2,\nu_3,\nu_4}$}
Consider the hyper-Hermitian structure on $M^{4n+4}_{\nu_1,\nu_2,\nu_3,\nu_4}$ as defined in Theorem \ref{extendedhypercomplextheorem}. We claim that it is in fact balanced: 
\begin{Prop}
	$(M^{4n+4}_{\nu_1,\nu_2,\nu_3,\nu_4}, \check{g}, a\check{\om}_1+b\check{\om}_2+c\check{\om}_3)$  is a balanced Hermitian manifold for $(a,b,c) \in S^2.$
\end{Prop}
First, a simple calculation shows that
\begin{align*}
	d(\check{\om}_1^{2n+1}) &=\ (2n+1)(\sigma_1^{2n} \w d(\nu_{12} +\nu_{34}))+n(2n+1) (\sigma_1^{2n-1}\w d(\nu_{1234})).
\end{align*}
Since $d\nu_i \in \mathfrak{sp}(n)$ the first term always contains an element in $\Lm^{2n+1,2n+1}(M^{4n})$ and thus is zero. The second term also vanishes since
\[\frac{1}{(2n-1)!}\sigma_1^{2n-1}\w d\nu_i = g_M(\sigma_1,d\nu_i)\vol_M=0.\]
Repeating this argument with $\check{\om}_2$ and $\check{\om}_3$ we conclude that the hyper-Hermitian structure of Theorem \ref{extendedhypercomplextheorem} is balanced. Note that by contrast the examples arising from Theorem \ref{hypercomplextheorem} are never balanced.

\bibliography{biblioG}
\bibliographystyle{plain}

\end{document}